\documentclass{article}
\usepackage{theorem}
\usepackage{amsmath,amssymb}  
\usepackage{bm}  
\usepackage{multirow}
\usepackage{color}

\newtheorem{theorem}{Theorem}
\newtheorem{proposition}[theorem]{Proposition}
\newtheorem{lemma}[theorem]{Lemma}
\newtheorem{cor}[theorem]{Corollary}

{\theorembodyfont{\normalfont\rmfamily}
\newtheorem{definition}[theorem]{Definition}

\newtheorem{remark}[theorem]{Remark}

} \makeatletter 

\makeatletter
\def\operatorname#1{\mathop{\operator@font #1}\nolimits}%
\makeatother
\newcommand{\s}{s}
\newcommand{\R}{\mathbb{R}}
\newcommand{\C}{\mathbb{C}}

\newcommand{\half}{{\tfrac{1}{2}}}

\newcommand{\Id}{\operatorname{Id}}
\newcommand{\id}{\operatorname{Id}}

\newcommand{\Sp}{\operatorname{Sp}}

\newcommand{\card}{\operatorname{Card}}

\newcommand{\Ker}{\operatorname{Ker}}
\newcommand{\im}{\operatorname{Im}}

\newenvironment{proof}[1][{}]{{ \textsc{Proof{#1}:~}}}{{\hspace*{\fill}$\square$\\}}

\def\adots{\mathinner{\mkern2mu\raise 1pt\hbox{.}\mkern 3mu\raise
4pt\hbox{.}\mkern2mu\raise 8pt\hbox{{.}}}}

\title{Normal Forms for  Symplectic Matrices}
\author{Jean Gutt \\
	\small\hbox{\parbox[t]{1.9in}{\begin{center}
	D{\'e}partement de Math{\'e}matique \\
	Universit{\'e} Libre de Bruxelles \\
	Campus Plaine, CP 218 \\
	Boulevard du Triomphe \\
	B-1050 Bruxelles \\
	Belgium\\
	jeangutt\char64ulb.ac.be\end{center}}}
	\hbox{\parbox[t]{.6in}{\begin{center}\rm and \end{center}}}
	\hbox{\parbox[t]{1.9in}{\begin{center}
	Universit\'e de Strasbourg\\
	IRMA\\
	7 rue Ren\'e Descartes\\
	67000 Strasbourg\\
	France\\
	{gutt\char64math.unistra.fr}\end{center}}}
	}
\date{} 
\begin{document} 

\maketitle

\begin{abstract}
	We give  a self contained and elementary description
	of normal forms for symplectic matrices,  based on  geometrical considerations.
	The normal forms in question are expressed in terms of elementary Jordan 
	matrices and integers with values in $\{ -1,0,1\}$ related to signatures of 
	quadratic forms naturally associated to the symplectic matrix.
\end{abstract}

\section*{Introduction}

Let  $V$ be a real vector space of dimension $2n$ with a non degenerate skewsymmetric bilinear form $\Omega$.
The symplectic group $\Sp(V,\Omega)$ is the set of linear transformations of $V$ which preserve $\Omega$:
\[
	\Sp(V,\Omega) = \left\{ \, A: V \rightarrow V\,\vert\, A \mbox{ linear and }\Omega(Au,Av)=\Omega(u,v) \, {\textrm{for all }}  u,v  \in V\,\right\}.
\]
	A \emph{symplectic basis} of the  symplectic vector space $(V,\Omega)$ of dimension $2n$ is a basis $\{ e_1, \ldots ,e_{2n} \}$
	in which the matrix representing the symplectic form is $\Omega_0 = 
	\left(\begin{smallmatrix}
		0 & \Id\\
		-\Id & 0
	\end{smallmatrix}\right)$.
In a symplectic basis, the matrix  $A'$ representing an element $A\in \Sp(V,\Omega)$ belongs to 
$$
\Sp(2n,\R)= \left\{ \, A'\in Mat(2n\times 2n,\R)\,\vert\, A^{'\tau} \Omega_0 A'=\Omega_0\,\right\}
$$
where $(\cdot)^\tau$ denotes the transpose of a matrix.\\ 
Given an element $A$ in the symplectic group $\Sp(V,\Omega)$, we want to find a symplectic basis of $V$ in which the matrix $A'$ representing $A$ has a  distinguished form; to give a \emph{normal form} for matrices in $\mathrm{Sp}(2n,\mathbb R)$ means to describe a distinguished representative in each conjugacy class.
In general, one cannot find a symplectic basis of the complexified vector space for which the matrix representing  $A$ has Jordan normal form. 

The normal forms considered here are expressed in terms of elementary Jordan matrices and matrices depending on an integer $\s \in \{-1,0,1\}$.
They are closely related to the forms given by Long in \cite{LD00, Lon02} ;
the main difference is that,  in those references, some indeterminacy was left in the choice of matrices in each conjugacy class, in particular when the matrix admits $1$ as an eigenvalue.
We speak in this case of \emph{quasi-normal forms}.
Other constructions can be found in \cite{Wim91, LM74, LinMehrXu99, SPen72, MulThi99} but they are either quasi-normal or far from Jordan normal forms. Closely related are the constructions of normal forms
for real matrices that are selfadjoint, skewadjoint or unitary with respect to an indefinite inner product where sign characteristics are introduced; they have been  studied in many sources; for instance -mainly for selfadjoint  and skewadjoint matrices- in the monograph of I. Gohberg, P. Lancaster and L. Rodman \cite{GLR}, and for unitary matrices in the papers \cite{YLR,GR,Mehl,Rod}.
 Normal forms for symplectic matrices have been  given by C. Mehl in \cite{Mehl2} and by  V. Sergeichuk in \cite{Serg} ; 
 in those descriptions, the basis producing the normal form is not required to be symplectic.
 
We construct here normal forms using elementary geometrical methods. 

The choice of representatives for normal (or quasi normal) forms of matrices depends on the application one has in view.
Quasi normal forms were used by Long to get precise formulas for indices of iterates of Hamiltonian orbits in \cite{Long}.
The forms obtained here were useful for us to give new characterisations of Conley-Zehnder indices of general paths of symplectic matrices \cite{JGutt}. We have chosen to give a normal form in
a symplectic basis. The main interest of our description is the natural interpretation of the signs appearing in the decomposition, and the description of the decomposition
for  matrices with $1$ as an eigenvalue. It also yields
an  easy natural characterization of   the conjugacy class of an element in $\Sp(2n,\R)$.
We hope it can be useful in other situations.\\

Assume that $V$ decomposes as a direct sum $V=V_1\oplus V_2$ where $V_1$ and $V_2$ are $\Omega$-orthogonal $A$-invariant subspaces.
Suppose that $\{e_1,\ldots,e_{2k}\}$ is a symplectic basis of $V_1$ in which the matrix representing $A\vert_{V_1}$ is
$A'=\left(\begin{smallmatrix}
	A'_1 & A'_2\\
	A'_3 & A'_4
\end{smallmatrix}\right)$.
Suppose also that $\{f_1,\ldots,f_{2l}\}$ is a symplectic basis of $V_2$ in which the matrix representing $A\vert_{V_2}$ is
$A''=\left(\begin{smallmatrix}
	A''_1 & A''_2\\
	A''_3 & A''_4
\end{smallmatrix}\right)$.
Then $\{e_1,\ldots,e_k,f_1,\ldots,f_l,e_{k+1},\ldots,e_{2k},f_{l+1},\ldots,f_{2l}\}$ is a symplectic basis of $V$ and the matrix representing $A$ in this basis is
$$
	\left(\begin{array}{cccc}
		A'_1 & 0 & A'_2 & 0\\
		0 & A''_1 & 0 & A''_2\\
		A'_3 & 0 & A'_4 & 0\\
		0 & A''_3 & 0 & A''_4
	\end{array}\right).
$$
The notation $A'\diamond A''$ is used in Long \cite{Long} for this matrix.
It is ``a direct sum of matrices with obvious identifications".
We call it the \emph{symplectic direct sum} of the matrices $A'$ and $A''$.\\

We $\C$-linearly  extend  $\Omega$ to  the complexified vector space $V^\C$ and we $\C$-linearly extend any $A\in \Sp(V,\Omega)$ to  $V^\C.$ 
If $v_\lambda$ denotes an eigenvector of $A$ in $V^\C$ of the eigenvalue $\lambda$, then
$\Omega(Av_\lambda,Av_\mu)=\Omega(\lambda v_\lambda,\mu v_\mu)=\lambda\mu \Omega(v_\lambda,v_\mu),$ thus
$\Omega(v_\lambda,v_\mu)=0$ unless $\mu=\frac{1}{\lambda}.$
Hence the eigenvalues of $A$ arise in ``quadruples''
\begin{equation}
	[\lambda]:=\left\{\lambda,\frac{1}{\lambda},\overline{\lambda},\frac{1}{\overline\lambda}\right\}.
\end{equation}
We find a symplectic basis of $V^\C$ so that $A$ is a symplectic direct sum of block-upper-triangular matrices of the form
$$
				\left(\begin{array}{cc}
					J({\lambda},k)^{-1}&0\\
					0& J({\lambda},k)^{\tau}
				\end{array}\right)
				\left(\begin{array}{cc}
					\Id&D(k,\s)\\
					0& \Id
				\end{array}\right),
$$
 or
$$
\left(\begin{smallmatrix}
					J(\overline{\lambda},k)^{-1}&&&{\text{\large{0}}}\\
					&J({\lambda},k)^{-1}&&\\
					&& J(\overline{\lambda},k)^{\tau}&\\
					{\text{\large{0}}}&&& J({\lambda},k)^{\tau}
				\end{smallmatrix}\right)
				\left(\begin{smallmatrix}
					\,\Id\,&0&0&D(k,\s)\\
					&\,\,\,\, \Id\,\,&D(k,\s)&0\\
					&&\Id&0\\
					{\text{\large{0}}}&&& \Id\,
				\end{smallmatrix}\right),
$$
or
$$
\left(\begin{smallmatrix}
					J(\overline{\lambda},k)^{-1}&&&{\text{\large{0}}}\\
					&J({\lambda},k+1)^{-1}&&\\
					&& J(\overline{\lambda},k)^{\tau}&\\
					{\text{\large{0}}}&&& J({\lambda},k+1)^{\tau}
				\end{smallmatrix}\right)
				\left(\begin{smallmatrix}
					\,\Id\,&0&0&S(k,\s,\lambda)\\
					&\,\,\,\,\Id\,\,&S(k,\s,\lambda)^{\tau}&0\\
					&&\Id&0\\
					{\text{\large{0}}}&&& \Id
				\end{smallmatrix}\right).
$$
Here, $J({\lambda},k)$ is  the elementary $k\times k$ Jordan matrix corresponding to an eigenvalue $\lambda$, $D(k,\s)$ is the diagonal $k\times k$ matrix
\[
	D(k,\s)=\operatorname{diag}(0,\ldots 0,\s),
\]
and $S(k,\s,\lambda)$ is  the $k\times(k+1)$ matrix defined by
\[
	S(k,\s,\lambda):=
	\left(\begin{smallmatrix} 
		0&\ldots& 0&0&0\\
		\vdots &&\vdots&\vdots&\vdots \\
		0&\ldots& 0&0&0\\
		0&\ldots &0&\half is&\lambda is\\
	\end{smallmatrix}\right),
\]
with $\s$ an integer in $\{ -1,0,1\}$.
Each $s\in\{\pm1\}$ is called a sign and the collection of such signs appearing in the decomposition of a matrix $A$ is called the sign characteristic of $A$.

More precisely, on the real vector space $V$, we shall prove:

\begin{theorem}[Normal forms for symplectic matrices]\label{thm:normalform}
	Any symplectic endomorphism $A$ of a finite dimensional symplectic vector space $(V,\Omega)$ is the  direct sum of its restrictions
	$A_{\vert V_{[\lambda]}}$ to the real $A$-invariant symplectic subspace  $V_{[\lambda]}$  whose complexification is
	the direct sum of the generalized eigenspaces of eigenvalues $\lambda,\tfrac{1}{\lambda},\overline{\lambda}$ and $\tfrac{1}{\overline\lambda}$:
	$$
		V_{[\lambda]}^{\C} := E_\lambda\oplus E_{\frac{1}{\lambda}}\oplus E_{\overline\lambda}\oplus E_{\frac{1}{\overline\lambda}}.
	$$
	We  distinguish three cases : $\lambda\notin S^1$, $\lambda=\pm 1$ and $\lambda\in S^1\setminus \{\pm 1\}$.
	
	{\bf{Normal form for $A_{\vert V_{[\lambda]}}$ for $\lambda\notin S^1:$}}\\
	Let $\lambda\notin S^1$ be an eigenvalue of $A$. Let $k:=\dim_\C \Ker (A-\lambda\Id)$ (on $V^\C$) and $q$
	 be the smallest integer so that $(A-\lambda\Id)^{q}$ is identically zero on the generalized eigenspace $E_\lambda$.
	\begin{itemize}
		\item If $\lambda$ is a real eigenvalue of $A$ ($\lambda\notin S^1$ so $\lambda\neq \pm1$),
			there exists a symplectic basis of $V_{[\lambda]}$
			in which the matrix representing the restriction of $A$ to $V_{[\lambda]}$ is a symplectic direct sum of $k$
			matrices of the form 
			$$
				\left(\begin{array}{cc}
					J({\lambda},q_j)^{-1}&0\\
					0& J({\lambda},q_j)^{\tau}
				\end{array}\right)
			$$
			with $q=q_1\ge q_2\ge \dots \ge q_k$   and $J(\lambda,m)$ is the elementary $m\times m$ Jordan matrix associated to $\lambda$
			\begin{equation*}
				J(\lambda,m)=
				\left(\begin{smallmatrix}
					\lambda & 1 &  &  &  &  & \\
					 & \lambda & 1 &  &  &\makebox(0,0) {\rm{0}}  & \\
					 &  & \lambda & 1 &  &  & \\
					 &  &  & \ddots & \ddots &  & \\
					 &\makebox(0,0){\rm{0}}  &  &  & \lambda & 1 & \\
					 &  &   &  &  & \lambda & 1\\
					 &  &  &  &  &  & \lambda
				\end{smallmatrix}\right).
			\end{equation*}
			This decomposition is unique, when $\lambda$ has been
			chosen in $\{ \lambda,\lambda^{-1}\}$. It is determined by the chosen $\lambda$ and by 	
			the dimension $ \dim\bigl(\Ker(A-\lambda\id)^r\bigr)$ for each $r>0$.

		\item If $\lambda=re^{i\phi}\notin(S^1\cup \R)$ is a complex eigenvalue of $A$, there exists a symplectic basis of $V_{[\lambda]}$
			in which the matrix representing the restriction of $A$ to $V_{[\lambda]}$ is a symplectic direct sum of $k$
			matrices of the form
			$$
				\left(\begin{array}{cc}
					J_\R\bigl(re^{-i\phi},2q_j\bigr)^{-1}&0\\
					0&J_\R\bigl(re^{-i\phi},2q_j\bigr)^\tau
				\end{array}\right)
			$$
			with $q=q_1\ge q_2\ge \dots \ge q_k$ and $J_\R(re^{i\phi},k)$ is the $2m\times 2m$ block upper triangular matrix defined by
			\begin{equation*}
				J_\R(re^{i\phi},2m):=
				\left(\begin{smallmatrix}
					R(re^{i\phi}) & \Id &  &  &  &  &\\
					 &R(re^{i\phi}) & \Id &  &  &\makebox(0,0){\rm{0}}  & \\
					 &  & R(re^{i\phi})& \Id &  &  & \\
					 &  &  & \ddots & \ddots &  & \\
					 &\makebox(0,0){\rm{0}}  &  &  & R(re^{i\phi}) &  \Id & \\
					 &  &  &  &  &R(re^{i\phi}) &  \Id \\
					 &  &  &  &  &  & R(re^{i\phi})
				\end{smallmatrix}\right)
			\end{equation*}
 			with $R(re^{i\phi})=\left(\begin{smallmatrix}
					r\cos \phi&-r\sin \phi\\
					r\sin \phi&r\cos \phi
			\end{smallmatrix}\right)$.\\
	This decomposition is unique, when $\lambda$ has been
			chosen in $\{ \lambda,\lambda^{-1}, \overline{\lambda},\overline{\lambda}^{-1}\}$. It is
			 is determined  by the chosen $\lambda$ and  by the dimension $ \dim\bigl(\Ker(A-\lambda\id)^r\bigr)$  for each $r>0$.
       \end{itemize}
	
	{\bf{Normal form for $A_{\vert V_{[\lambda]}}$ for $\lambda=\pm 1:$}}\\
	Let $\lambda=\pm 1$ be an eigenvalue of $A$. 
	There exists a symplectic basis of $V_{[\lambda]}$
	in which the matrix representing the restriction of $A$ to $V_{[\lambda]}$ is a symplectic direct sum of 
	matrices of the form 
	$$
		\left(\begin{array}{cc}
			J(\lambda,r_j)^{-1}&C(r_j,\s_j,\lambda)\\
				0& J(\lambda,r_j)^{\tau}
		\end{array}\right)
	$$
	where $C(r_j,\s_j,\lambda):=J(\lambda,r_j)^{-1} \operatorname{diag}\bigl(0,\ldots,0, \s_j\bigr)$  with $\s_j\in \{0,1,-1\}$.
	If $\s_j=0$,  then $r_j$ is odd.
	The dimension of the eigenspace of  the eigenvalue $\lambda$ is given by $2\card \{j \,\vert\, \s_j=0\}+\card\{j \,\vert \,\s_j\neq0\}$. \\
The  number of $s_j$ equal to $+1$ (resp. $-1$) arising in blocks of dimension $2k$ (i.e. with corresponding $r_j=k$) is equal to the number of positive (resp. negative) eigenvalues of the symmetric $2$-form 
$$\hat{Q}^\lambda_{2k}:\Ker\bigl( (A-\lambda\id)^{2k}\bigr)\times \Ker\left( (A-\lambda\id)^{2k}\right)\rightarrow \R$$
$$(v,w)\mapsto \lambda\, \Omega\bigl((A-\lambda\id)^{k}v,(A-\lambda\id)^{k-1}w\bigr).$$	
The decomposition  is unique up to a permutation of the blocks and
 is  determined by $\lambda$,  by the dimension
$\dim\bigl( \Ker (A-\lambda\Id)^r\bigr)$ for each $r\ge 1$, and  by the rank and the signature of  the
symmetric bilinear $2$-form  $\hat{Q}^\lambda_{2k}$ for each $k\ge 1$. 
	
	{\bf{Normal form for $A_{\vert V_{[\lambda]}}$ for $\lambda\in S^1\setminus \{\pm 1\}:$}}\\
	Let $\lambda \in S^1, \lambda \neq\pm 1$ be an eigenvalue of $A$. 
	There exists a symplectic basis of $V_{[\lambda]}$
	in which the matrix representing the restriction of $A$ to $V_{[\lambda]}$ is a symplectic direct sum of 
	$4k_j\times 4k_j$ matrices ($k_j\ge 1$) of the form 
	\begin{equation}
		\makeatletter
			\setbox\strutbox\hbox{%
			\vrule\@height.5\baselineskip
			\@depth.2\baselineskip
			\@width\z@}
		\makeatother
		\left(\begin{smallmatrix}
			\bigl(J_\R(\overline{\lambda},2k_j)\bigr)^{-1}&\vline&
				\begin{smallmatrix} 
					\begin{smallmatrix}
						0\\
						\vdots\\
						0\strut \\
					\end{smallmatrix}
					&
					\begin{smallmatrix}
						\cdots\\
						\phantom{\vdots}\\
						\cdots\strut
					\end{smallmatrix}&
					\begin{smallmatrix}
						0\\
						\vdots\\
						0\strut \\
					\end{smallmatrix}
					&\s_j\, V_{k_j}^1(\phi) &\s_j\, V_{k_j}^2(\phi)\\
				\end{smallmatrix})\\
			\hline 
			0 \strut&\vline&\strut \bigl(J_\R(\overline{\lambda},2k_j)\bigr)^\tau\strut\\
		\end{smallmatrix}\right)
	\end{equation}

	and    $(4k_j+2)\times (4k_j+2)$ matrices ($k_j\ge 0$) of the form
\begin{equation}
	\makeatletter
		\setbox\strutbox\hbox{%
		\vrule\@height.5\baselineskip
		\@depth.2\baselineskip
		\@width\z@}
	\makeatother
	\left(\begin{smallmatrix}
		\bigl(J_\R(\overline{\lambda},2k_j)\bigr)^{-1}&\vline&\s_j\, U_{k_j}^2(\phi)&\vline&
			\begin{smallmatrix}
				0\strut\\
				\vdots\\ 
				0\strut
			\end{smallmatrix}
			&
			\begin{smallmatrix}
				\cdots\\
				\phantom{\vdots}\\
				\cdots
			\end{smallmatrix}&
			\begin{smallmatrix}
				0\strut\\
				\vdots\\
				0\strut
			\end{smallmatrix}
			&\frac{\s_j}{2} V_{k_j}^2(\phi) &\frac{-\s_j}{2}V_{k_j}^1(\phi)&\vline &U_{k_j}^1(\phi)\\
		\hline
		0\strut&\vline&\cos\phi&\vline&
			0&\ldots & 0& 1 & 0
		&\vline&\s_j\sin\phi\\
		\hline
		0
		&\vline&\begin{smallmatrix}
			0\strut\\
			\vdots\\
			0\strut\\
		\end{smallmatrix}
		&\vline& &&&\bigl(J_\R(\overline{\lambda},2k_j)\bigr)^\tau &&\vline&
		\begin{smallmatrix}
			0\strut\\
			\vdots\\
			0\\
		 \end{smallmatrix} \\
		 \hline
		0\strut&\vline&-\s_j\sin\phi&\vline&
		0&\ldots & 0&0  &-\s_j&\vline
		&\cos\phi\\
	\end{smallmatrix}\right)
\end{equation}
	where $J_\R(e^{i\phi},2k)$ is defined as  above,
	where $  \left(V_{k_j}^1 (\phi) \, V_{k_j}^2(\phi) \right)$ is the $2k_j\times 2$ matrix
	defined by 
	\begin{equation}
		\left(V_{k_j}^1 (\phi) \, V_{k_j}^2(\phi) \right)=
		\left(\begin{matrix} (-1)^{k_j-1}  R(e^{i k_j\phi})\\
			 \vdots\\R(e^{i\phi})
		 \end{matrix}\right)
	\end{equation}
	with $R(e^{i\phi})=\left(\begin{smallmatrix}
		\cos \phi&-\sin \phi\\
		\sin \phi&\cos \phi
	\end{smallmatrix}\right)$, where  
	\begin{equation}
		\left(U^1_{k_j} (\phi) \, U^2_{k_j}(\phi) \right)= \left(V^1_{k_j} (\phi) \, V^2_{k_j}(\phi) \right)\left(R(e^{i\phi})\right)
	\end{equation}
	and  where $\s_j=\pm 1$.
	The complex dimension of the eigenspace of the eigenvalue $\lambda$ in $V^\C$ is given by the number of such matrices.\\
	The number of $s_j$ equal to $+1$ (resp. $-1$) arising in blocks of dimension $2m$  in the normal decomposition
	given above is equal to the number of positive (resp. negative) eigenvalues of the Hermitian $2$-form $\hat{Q}^\lambda_m$ defined on $\Ker\bigl( (A-\lambda\id)^{m}\bigr)$ by:
\[\begin{array}{lll}
\hat{Q}^\lambda_{m} :& \Ker\bigl( (A-\lambda\id)^{m}\bigr)\times \Ker\bigl( (A-\lambda\id)^{m}\bigr)\rightarrow \C\nonumber&\\
&\quad (v,w)\mapsto\frac{1}{ \lambda} \Omega\bigl((A-\lambda\id)^{k}v,(A-\overline{\lambda}\id)^{k-1}\overline{w}\bigr)  &\textrm { if }m=2k\\
&\quad (v,w)\mapsto  i\, \Omega\bigl((A-\lambda\id)^{k}v,(A-\overline{\lambda}\id)^{k}\overline{w}\bigr)  &\textrm { if }m=2k+1.
\end {array}\]
This decomposition is unique up to a permutation of the blocks, when $\lambda$ has been
			chosen in $\{ \lambda, \overline{\lambda} \}$.
			It is  determined by the chosen $\lambda$, by the dimension
$\dim\bigl( \Ker (A-\lambda\Id)^r\bigr)$ for each $r\ge 1$ and  by the rank and the signature of  the
Hermitian bilinear $2$-form  $\hat{Q}^\lambda_m$ for each $m\ge 1$. 
\end{theorem}

The normal form for $A_{\vert V_{[\lambda]}}$ is given in Theorem \ref{thm:normal1} for $\lambda\notin S^1$, in Theorem \ref{normalforms1} for $\lambda=\pm1$, and in Theorem \ref{normalforms3} for $\lambda \in S^1\setminus\{\pm1\}$.
The characterisation of the signs is given in Proposition \ref{sumd} for $\lambda=\pm1$ and in Proposition \ref{sumds} for $\lambda \in S^1\setminus\{\pm1\}$.

A direct consequence of Theorem \ref{thm:normalform} is the following characterization of the conjugacy class of a matrix in the symplectic group.
\begin{theorem}
	The conjugacy class of a matrix $A\in\Sp(2n,\R)$ is determined by the following data:
	\begin{itemize}
		\item the eigenvalues of $A$ which arise in quadruples $[\lambda]=\{ \lambda,\lambda^{-1},\overline{\lambda},\overline{\lambda}^{-1}\}$;
\item the dimension $\dim\bigl( \Ker (A-\lambda\Id)^r\bigr)$ for each $r\ge 1$ for one eigenvalue in each class $[\lambda]$;
\item for $\lambda=\pm 1$, the rank and the signature of the symmetric form $\hat{Q}^\lambda_{2k}$ for each $k\ge 1$ and  for an eigenvalue $\lambda$ in $S^1\setminus\{\pm 1\}$ chosen in each $[\lambda]$, the rank and the signature of the Hermitian form $\hat{Q}^\lambda_{m}$ for each $m\ge 1$, with 
\[\begin{array}{lll}
\hat{Q}^\lambda_{m} :& \Ker\bigl( (A-\lambda\id)^{m}\bigr)\times \Ker\bigl( (A-\lambda\id)^{m}\bigr)\rightarrow \C\nonumber&\\
&\quad (v,w)\mapsto\frac{1}{ \lambda} \Omega\bigl((A-\lambda\id)^{k}v,(A-\overline{\lambda}\id)^{k-1}\overline{w}\bigr)  &\textrm { if }m=2k\\
&\quad (v,w)\mapsto  i\, \Omega\bigl((A-\lambda\id)^{k}v,(A-\overline{\lambda}\id)^{k}\overline{w}\bigr)  &\textrm { if }m=2k+1.
\end {array}\]

\end{itemize}\hfill{$\square$}
 \end{theorem}

\section*{Acknowledgements}
I thank my thesis supervisors, Fr\'ed\'eric Bourgeois and Alexandru Oancea who  encouraged me to write this text.
I deeply thank the referee for the many  improvements he brought to this text, and  for pointing to me the works on normal forms of matrices in the framework of indefinite inner product spaces \cite{YLR,GLR, GR,Mehl,Rod},
and more particularly the normals forms for symplectic matrices given in \cite{Mehl2, Serg}.
I am grateful to the  Foundation for Scientific Research (FNRS-FRS) for its support.
The author acknowledges partial support from the ERC via the grant StG-259118-STEIN, from an ARC
of the Communaut\'e fran\c caise de Belgique and from the P\^ole d'Attraction Interuniversitaire Dygest.


\section{Preliminaries}

\begin{lemma}\label{lem:orthokerim}
	Consider $A\in\Sp(V,\Omega)$ and let $0\neq \lambda\in \C$.
	Then $\Ker(A-\lambda\Id)^j$ in $V^\C$ is the symplectic orthogonal complement of $\im(A- \tfrac{1}{\lambda}\Id)^j.$
\end{lemma}
\begin{proof}
	\begin{eqnarray*}
		\Omega \bigl( (A - \lambda \Id ) u, Av \bigr) & = & \Omega(Au,Av)-\lambda\Omega(u,Av)
		=\Omega(u,v)-\lambda\Omega(u,Av)\\
		&=&-\lambda\Omega \Bigl( u, \bigl(A-\tfrac{1}{\lambda} \Id \bigr) v \Bigr)
	\end{eqnarray*}
	and by induction
	\begin{equation}\label{eq:omegaA}
		\Omega \bigl( (A-\lambda\Id)^ju, A^jv \bigr) = (-\lambda)^j\Omega \Bigl( u, \bigl( A-\tfrac{1}{\lambda} \Id \bigr)^jv \Bigr).
	\end{equation}
	The result follows from the fact that $A$ is invertible.
\end{proof}
\begin{cor}\label{lem:orthovalpro}
	If $E_\lambda$ denotes the generalized eigenspace of eigenvalue $\lambda,$ i.e $E_\lambda := \bigl\{ v\in V^\C \ \vert \ (A-\lambda\Id)^jv=0 \textrm{ for an integer } j >0 \bigr\}$, we have
	\begin{equation*}
		\Omega(E_\lambda,E_\mu)=0 \quad \textrm{ when }~\lambda\mu\ne 1 .
	\end{equation*}
\end{cor}
Indeed the symplectic orthogonal complement of $E_\lambda=\cup_j \Ker (A-\lambda\Id)^j$ is the intersection of the $\im(A-\tfrac{1}{\lambda}\Id)^j.$ By Jordan normal form, this intersection is
the sum of the generalized eigenspaces corresponding to the eigenvalues which are not  $\frac{1}{\lambda}.$\\

If $v= u+iu'$ is in $\Ker (A-\lambda\Id)^j$ with $u$ and $u'$ in $V$ then
$\overline{v}=u-iu'$ is in $\Ker (A-\overline{\lambda}\Id)^j$ so that  $E_\lambda \oplus E_{\overline{\lambda}}$ is the complexification of a real subspace of $V$.
From this remark and corollary \ref{lem:orthovalpro} the space
\begin{equation}
	W_{[\lambda]}:=E_\lambda\oplus E_{\frac{1}{\lambda}}\oplus E_{\overline\lambda}\oplus E_{\frac{1}{\overline\lambda}}
\end{equation}
is  the complexification of a real and symplectic $A$-invariant subspace $V_{[\lambda]}$   and
\begin{equation}
	V=V_{[\lambda_1]}\oplus V_{[\lambda_2]}\oplus\ldots\oplus V _{[\lambda_K]}
\end{equation}
where  we  denote by
$\left[ \lambda\right]$ the set  $\{ \lambda, \overline\lambda, \frac{1}{\lambda},\frac{1}{\overline\lambda}\}$ and by $\left[\lambda_1\right], \ldots,\left[ \lambda_K\right] $
the distinct such sets exhausting the eigenvalues of $A$.\\
We denote by $A_{[\lambda_i]}$  the restriction of $A$ to $V_{[\lambda_i]}.$ It is clearly enough to obtain normal forms for each $A_{[\lambda_i]}$ since $A$ will be a symplectic direct sum of those.

We shall construct a symplectic basis of $W_{[\lambda]}$ (and of $V_{[\lambda]}$) adapted to $A$ for a given eigenvalue $\lambda$ of $A$.
We assume that  $(A-\lambda\Id)^{p+1}=0$ and $(A-\lambda\Id)^{p}\ne 0$ on the generalized eigenspace $E_\lambda.$ Since $A$ is real,
this integer $p$ is the same for $\overline{\lambda}$.
By lemma \ref{lem:orthokerim}, $\Ker (A-\lambda\Id)^j$ is the symplectic orthogonal complement of $\im\bigl(A-\tfrac{1}{\lambda}\Id\bigr)^j$ for all $j,$ thus $\dim\Ker (A-\lambda\Id)^j=\dim \Ker \bigl(A-\tfrac{1}{\lambda}\Id\bigr)^j$;
hence the integer $p$ is the same for $\lambda$ and $\frac{1}{\lambda}$.

We decompose $W_{[\lambda]}$ (and $V_{[\lambda]}$) into a direct sum of $A$-invariant symplectic subspaces.
Given a symplectic subspace $Z$ of $V_{[\lambda]}$ which is $A$-invariant  , its orthogonal complement  (with respect to the symplectic $2$-form) $V':=Z^{\perp_\Omega}$ is again symplectic and $A$-invariant. The generalized eigenspace for $A$ on $V^{'\C}$  are $E'_{\mu}=V^{'\C}\cap E_{\mu}$, 
and the smallest integer $p'$ for which $(A-\lambda\Id)^{p'+1}=0$ on $E'_{\lambda}$ is such that $p'\le p$.

Hence, to get the decomposition  of $W_{[\lambda]}$ (and $V_{[\lambda]}$) it is enough to build
a symplectic subspace  of $W_{[\lambda]}$ which is $A$-invariant  and closed under complex conjugation
and to proceed inductively. We shall construct such a subspace, containing a well chosen vector $v\in E_\lambda$ so that $(A-\lambda\Id)^{p}v \neq 0$.

We shall distinguish three cases; first $\lambda\notin S^1$ then $\lambda=\pm 1$ and
finally $\lambda\in S^1\setminus \{\pm 1\}.$

We first present a few technical lemmas which will be used for this construction.

\subsection{A few technical lemmas}\label{appendicesemisimple}
Let $(V,\Omega)$ be a real symplectic vector space. Consider $A\in\Sp(V,\Omega)$ and let $\lambda$ be an eigenvalue of $A$ in $V^\C.$
	\begin{lemma}\label{lem:tech1}
	For any positive integer $j$, the bilinear map
	\begin{equation*}
		\widetilde{Q}_j : \raisebox{.2ex}{$E_\lambda$}\,/\raisebox{-.2ex}{$\Ker(A-\lambda\Id)^j$}\times \raisebox{.2ex}{$E_{\frac{1}{\lambda}}$}/\raisebox{-.2ex}{$\Ker\bigl(A-{\tfrac{1}{\lambda}} \Id\bigr)^j$}\rightarrow \C 
	\end{equation*}
	\begin{equation}
		\bigl([v],[w]\bigr)\mapsto  \widetilde{Q}_j \bigl([v],[w]\bigr):=\Omega \bigl( (A-\lambda\Id)^jv,w \bigr) \qquad v\in E_\lambda, w\in E_{\frac{1}{\lambda}}
	\end{equation}
	is well defined and non degenerate.
	In the formula, $[v]$ denotes the  class containing $v$ in the appropriate quotient.
\end{lemma}
\begin{proof}
 The fact that $\widetilde{Q}_j $ is well defined follows from equation \eqref{eq:omegaA};
 indeed, for any integer $j$, we have 
 \begin{equation}
	\Omega \bigl( (A-\lambda\Id)^j u,v \bigr) = (-\lambda)^j \Omega \Bigl( A^j u, \bigl( A-\tfrac{1}{\lambda}\Id \bigr)^jv \Bigr).
\end{equation}
The map is non degenerate because $\widetilde{Q}_j \bigl([v],[w]\bigr)=0$ for all $ w$ if and only if $(A-\lambda\Id)^j v=0$ since $\Omega$ is a non degenerate pairing between
$E_\lambda$ and $E_{\frac{1}{\lambda}}$, thus if and only if $[v]=0.$
	Similarly, $ \widetilde{Q}_j \bigl([v],[w]\bigr)=0$ for all $v $ if and only if $w$ is $\Omega$-orthogonal to $\im(A-\lambda\Id)^j,$ thus if and only if
	$w\in \Ker\bigl(A-\tfrac{1}{\lambda}\Id\bigr)^j$ hence $[w]=0.$
\end{proof}
\begin{lemma}
For any $v,w\in V$, any $\lambda\in \C\setminus\{0\}$ and any integers $i\ge 0,~j>0$ we have:
\begin{eqnarray}
 \Omega \Bigl( (A-\lambda\Id)^{i}v,\bigl(A-\tfrac{1}{\lambda}\Id\bigr)^{j}w \Bigr)&=& -\frac{1}{\lambda}
  \Omega \Bigl( (A-\lambda\Id)^{i+1}v,\bigl(A-\tfrac{1}{\lambda}\Id\bigr)^{j}w \Bigr)\label{eq:Aij}\\
  &&-\frac{1}{\lambda^2}  \Omega \Bigl( (A-\lambda\Id)^{i+1}v,\bigl(A-\tfrac{1}{\lambda}\Id\bigr)^{j-1}w \Bigr).\nonumber
\end{eqnarray}
In particular, if $\lambda$ is an eigenvalue of $A$, if $v\in E_\lambda$ is such that $p\ge 0$ is the largest integer  for which $(A-\lambda\Id)^{p}v\neq 0$,  we have  for any integers $k,j\ge 0$:
	\begin{equation}\label{eq:p}
		\Omega \bigl( (A-\lambda\Id)^{p+k}v,w \bigr) = (-\lambda^2)^j \Omega \Bigl( (A-\lambda\Id)^{p+k-j}v, \bigl(A-{\tfrac{1}{\lambda}}\Id \bigr)^jw \Bigr)
	\end{equation}
so that
	\begin{equation}\label{eq:Qbiendef}
		\Omega \bigl( (A-\lambda\Id)^pv,w \bigr)=(-\lambda^2)^p \Omega \Bigl(v, \bigl( A - {\tfrac{1}{\lambda}}\Id \bigr)^pw \Bigr)
	\end{equation}
and
	\begin{equation}\label{eq:zero}
		\Omega \Bigl( (A-\lambda\Id)^{k}v, \bigl( A-{\tfrac{1}{\lambda}}\Id \bigr)^jw\Bigr)=0 ~\textrm{if }~k+j>p.
	\end{equation}
\end{lemma}	
\begin{proof}
We have:
{\begin{eqnarray*}
	&&\Omega\Bigl((A-\lambda\Id)^{i} v,\bigl(A-\tfrac{1}{\lambda}\Id\bigr)^jw\Bigr)\\
	 &&\ \ \  =
	 -\frac{1}{\lambda}\Omega\Bigl(\bigl(A-{\lambda}\Id
	 -A\bigr)(A-\lambda\Id)^{i} v,\bigl(A-\tfrac{1}{\lambda}\Id\bigr)^jw\Bigr)\\
	&&\ \ \ =-\frac{1}{\lambda}\Omega\Bigl((A-\lambda\Id)^{i+1} v,\bigl(A-\tfrac{1}{\lambda}\Id\bigr)^jw\Bigr)\\
	&&\ \ \  \quad+\frac{1}\lambda \Omega\Bigl(A(A-\lambda\Id)^{i} v,\bigl(A-\tfrac{1}{\lambda}\Id\bigr)\bigl(A-\tfrac{1}{\lambda}\Id\bigr)^{j-1}w\Bigr)\\
	&&\ \ \ =-\frac{1}{\lambda}\Omega\Bigl((A-\lambda\Id)^{i+1} v,\bigl(A-\tfrac{1}{\lambda}\Id\bigr)^jw\Bigr)\\
	&&\ \ \  \quad+\frac{1}{\lambda}\Omega\Bigl((A-\lambda\Id)^{i} v,\bigl(A-\tfrac{1}{\lambda}\Id\bigr)^{j-1}w\Bigr)\\
	&&\ \ \  \quad -\frac{1}{\lambda^2}\Omega\Bigl(A(A-\lambda\Id)^{i} v,\bigl(A-\tfrac{1}{\lambda}\Id\bigr)^{j-1}w\Bigr)
\end{eqnarray*}}
and  formula \eqref{eq:Aij} follows.\\
For any integers $k,j \ge 0$ and   any $v$ such that $(A-\lambda\id)^pv=0$, we have, by (\ref{eq:omegaA}), 
	\[
		(-\lambda)^{j} \Omega \Bigl( (A-\lambda\Id)^{p+k+1-j}v, \bigl(A-{\tfrac{1}{\lambda}}\Id \bigr)^{j} w \Bigr)=\Omega \bigl( (A-\lambda\Id)^{p+k+1}v,A^{j}w \bigr)=0.
	\]
	Hence, applying  formula \eqref{eq:Aij} with		
a decreasing  induction on $j$, we get formula \eqref{eq:p}. The other formulas follow readily.
	\end{proof}
\begin{definition}
For $\lambda \in S^1$  an eigenvalue of $A$ and $v\in E_\lambda$ a generalized eigenvector, we define
\begin{equation} \label{eq:defT}
T_{i,j}(v):=\frac{1}{\lambda^i {\overline{\lambda}}^j}\Omega\bigl((A-\lambda\id)^i v, (A-{\overline{\lambda}})^j\overline{v}\bigr).
\end{equation}
We have, by equation (\ref{eq:Aij}) : 
\begin{equation}\label{sumT}
T_{i,j}(v)=-T_{i+1,j}(v)-T_{i+1,j-1}(v),
\end{equation}
and also,
\begin{equation}\label{eq:skewT}
T_{i,j}(v)=-\overline{T_{j,i}(v)}.
\end{equation}
\end{definition}
\begin{lemma}\label{lem:Todd}
Let $\lambda \in S^1$ be an eigenvalue of $A$ and $v\in E_\lambda$ be a generalised eigenvector such that  the largest integer $p$ so that
$(A-\lambda\id)^pv \neq 0$ is odd, say, $p=2k-1$.  Then, in the $A$-invariant subspace $E^v_\lambda$ of $E_\lambda$ generated by $v$, there exists a vector $v'$
generating the same $A$-invariant subspace $E^{v'}_\lambda=E^v_\lambda$, so that $(A-\lambda\id)^pv' \neq 0$ and so that 
$$
T_{i,j}(v')=0 \quad\textrm{ for all } i,j\le k-1.
$$
If $\lambda$ is real (i.e. $\pm1$), and if $v$ is a real vector (i.e. in $V$),  the vector $v'$ can be chosen to be real as well.
\end{lemma}
\begin{proof}
Observe  that 
\begin{eqnarray*}
T_{k,k-1}(v)&=& -T_{k,k}(v)-T_{k-1,k}(v) \quad \textrm{by } \eqref{eq:Aij}\\
&=&-T_{k-1,k}(v)  \quad \textrm{by } \eqref{eq:zero}\\
&=&\overline{T_{k,k-1}(v)}  \quad \textrm{by } \eqref{eq:skewT}
\end{eqnarray*}
 is real and can be put to $d=\pm 1$ by rescaling the vector.
We use  formulas \eqref{eq:Aij}  and  \eqref{eq:skewT} and we proceed by decreasing induction on $i+j$  as follows:
\begin{itemize}
\item if $T_{{k-1},{k-1}}(v)= \alpha_1$, this $\alpha_1$ is purely imaginary, we replace $v$ by $$v':=v-\frac{\alpha_1}{2\lambda d}(A-\lambda\Id)v;$$
clearly  $E_\lambda^{v'}=E_\lambda^v$ and  $T_{i,j}(v')=T_{i,j}(v)$  for $i+j\ge 2k-1$ but now 
$$T_{{k-1},{k-1}}(v')= \alpha_1
-\frac{\alpha_1}{2d}T_{k,k-1}(v)-\frac{\overline{\alpha_1}}{2d}T_{k-1,k}(v)=0;$$ 
 so we can now assume $T_{{k-1},{k-1}}(v)=0$; observe that if $\lambda$ is real  and $v$ is in $V$, then $\alpha_1=0$ and $v'=v$;
\item if $T_{{k-2},{k-1}}(v)= \alpha_2=-T_{{k-1},{k-2}}(v)$, this $\alpha_2$ is real and we replace $v$ by $$v-\frac{\alpha_2}{2\lambda^2d}(A-\lambda\Id)^{2}v;$$ the space $E_\lambda^v$ does not change
and the quantities $T_{i,j}(v)$ do not vary for $i+j\ge 2k-2$;
now $$T_{{k-2},{k-1}}(v')=\alpha_2-\frac{\alpha_2}{2d}T_{{k},{k-1}}(v)-\frac{\overline{\alpha_2}}{2d}T_{{k-2},{k+1}}(v)=0,$$hence also 
$T_{{k-1},{k-2}}(v')= 0$; observe that if $\lambda$ is real  and $v$ is in $V$, then  $v'$ is in $V$.
\item we now assume by induction to have a $J>0$ so  that $T_{i,j}(v)=0$ for all $0\le i,j\le k-1$ so that $i+j> 2k-1-J$;
\item if $T_{{k-J},{k-1}}(v)= \alpha_J$, then $T_{{k-J},{k-1}}(v)=(-1)^{J-1}T_{{k-1},{k-J}}(v)$
so that $\alpha_J$ is real when $J$ is even and is imaginary when $J$ is odd; we replace $v$ by $$v-\frac{\alpha_J}{2\lambda ^Jd}(A-\lambda\Id)^{J}v;$$ the space $E_\lambda^v$ does not change and the quantities $T_{i,j}(v)$ do not vary for $i+j\ge 2k-J$; but
 now 
 \begin{eqnarray*}
 T_{{k-J},{k-1}}(v')&=& \alpha_J -\frac{\alpha_J}{2d}T_{{k},{k-1}}(v)-\frac{\overline{\alpha_J}}{2d}T_{{k-J},{k+J-1}}(v)\\
 &=&\alpha_J-\frac{\alpha_J}{2}-(-1)^J\frac{\overline{\alpha_J}}{2}=0.
 \end{eqnarray*}
  Hence also 
$T_{{k-J+1},{k-2}}(v')= 0,\ldots ~T_{{k-1},{k-J+1}}(v')= 0$; so the induction proceeds. Observe that if $\lambda$ is real and $v$ is in $V$ then $v'$ is in $V$.
\end{itemize}
\end{proof}
We shall use repeatedly  that a $n\times n$ block triangular  symplectic matrix is of the form
\begin{equation}\label{mattriang}
	A'= \left(\begin{array}{cc}
		B & C\\
		0& D
	\end{array}\right)
	\in \Sp(2n,\R) \Leftrightarrow
	\left\{\begin{array}{l}
		B=(D^{\tau})^{-1} \\
		C=(D^{\tau})^{-1} S \, \textrm{ with}\, S \, \textrm{ symmetric. }
	\end{array}\right.
\end{equation}

\section{Normal forms  for  $A_{\vert V_{[\lambda]}}$  when $\lambda\notin S^1 .$}

As before, $p$ denotes the largest integer such that $(A-\lambda\id)^{p}$ does not vanish identically
on the generalized eigenspace $E_\lambda$.
Let us choose an element  $v\in E_\lambda$ and an element $w\in E_{\frac{1}{\lambda}}$ such that 
$$
	\widetilde{Q}_p \bigl([v],[w]\bigr)=\Omega \bigl( (A-\lambda\Id)^pv,w \bigr)\ne 0.
	$$
Let us consider the smallest $A$-invariant subspace $E^v_\lambda$ of $E_\lambda$  containing $v$; it is of dimension $p+1$ and a basis is given by 
\begin{equation*}
	\bigl\{a_{0}:=v,\ldots, a_i:=(A-\lambda\Id)^{i}v,\ldots, a_p:=(A-\lambda\Id)^pv \bigr\}.
\end{equation*}
Observe that $Aa_i=(A-{\lambda}\Id)a_i+{\lambda} a_i$ so that $Aa_i={\lambda} a_i+ a_{i+1}$ for $i<p$ and $Aa_p=a_p$.

Similarly, we consider the smallest $A$-invariant subspace $E^w_{\frac{1}{\lambda}}$ of $E_{\frac{1}{\lambda}}$  containing $w$; it is also of dimension $p+1$  and a basis is given by 
\begin{equation*}
	\left\{b_0:=w, \ldots, b_{j}:=\left(A-\tfrac{1}{\lambda}\Id\right)^{j}w,\ldots b_{p}:=\left(A-\tfrac{1}{\lambda}\Id\right)^pw\right\} .
\end{equation*}
One has
\begin{itemize}
	\item[~] $\Omega(a_i,a_j)=0$ and $\Omega(b_i,b_j)=0$ because $\Omega(E_\lambda,E_\mu)=0$ if $\lambda\mu\ne 1$;
	\item[~] $\Omega(a_i,b_j)=0$ if $i+j>p$  by equation (\ref{eq:zero}) ;
         \item[~] $\Omega(a_i,b_{p-i})=\bigl(\frac{-1}{\lambda^2}\bigr)^{p-i}\Omega \Bigl(\bigl( A-{{\lambda}}\Id\bigr)^pv,w\Bigr)$ by equation (\ref{eq:p}) and is non zero by the choice of $v,w .$
\end{itemize}
The matrix representing $\Omega$	in the basis  $\{ b_p,\ldots,b_{0}, a_0,\ldots,a_{p}\}$ is thus of the form
\begin{equation*}
\left(\begin{smallmatrix}
\begin{smallmatrix}
	0&&0\\
	&\ddots&\\
	0&&0\\
\end{smallmatrix} &\,\vline\,&\begin{smallmatrix}
	\overline{\ast}&&0\\
	&\ddots&\\
	\ast&&\overline{\ast}\\
\end{smallmatrix}\\ ~\\ \hline \\
\begin{smallmatrix}
	\overline{\ast}&&\ast\\
	&\ddots&\\
	0&&\overline{\ast}
\end{smallmatrix}&\,\vline\,&\begin{smallmatrix}
	0&&0\\
	&\ddots&\\
	0&&0\\
\end{smallmatrix}
\end{smallmatrix}\right)
\end{equation*}
with non vanishing $\overline{\ast}$.
Hence $\Omega$ is non degenerate on  $E^v_\lambda\oplus E^w_{\frac{1}{\lambda}}$ which is thus a symplectic $A$-invariant subspace.

We now construct a symplectic  basis $\left\{b'_p,\ldots, b'_{0},a_0,\ldots,a_{p}\right\}$ of $E^v_\lambda\oplus E^w_{\frac{1}{\lambda}}$, 
extending  $\left\{a_0,\ldots,a_{p}\right\}$, using a Gram-Schmidt procedure on the $b_i$'s. This gives a normal form for $A$ on $E^v_\lambda\oplus E^w_{\frac{1}{\lambda}}$.

If $\lambda$ is real, we  take $v, w $ in  the real generalized eigenspaces $E^{\R}_\lambda$ and $E^{\R}_{\frac{1}{\lambda}}$ and we obtain a symplectic basis of the real $A$-invariant symplectic vector space ,
$E^{\R v}_\lambda\oplus E^{\R w}_{\frac{1}{\lambda}}$. 
If $\lambda$ is not real, one considers the basis of $E^{\overline{v}}_{\overline\lambda}\oplus E^{\overline{w}}_{\frac{1}{\overline{\lambda}}}$ defined by the conjugate vectors
$\{\overline{b'_p},\ldots,\overline{b'_0},\overline{a_0},\ldots, \overline{a_{p}}\}$ 
and this  yields a conjugate normal form on $E_{\overline{\lambda}}\oplus  E_{\frac{1}{\overline{\lambda}}}$, hence a normal form on $W_{[\lambda]} $ and this will induce a real normal form on
$V_{[\lambda]}$.\\

We choose $v$ and $w$ such that $\Omega \Bigl( \bigl(A-{\frac{1}{\lambda}}\Id\bigr)^pw,v \Bigr) = 1.$
We define inductively   on $j$
\begin{itemize}
\item[~] $b'_{p}:= \frac{1}{\Omega(b_p,a_0)}b_p=b_p $;
\item[~] $b'_{p-j}=\frac{1}{\Omega(b_{p-j},a_j)} \bigl( b_{p-j}-\sum_{k<j}\Omega(b_{p-j}, a_k) b'_{p-k}\bigr),$\\
so that any $b'_j$ is a linear combination of the $b_r$ with $r\ge j$.
\end{itemize}	
In the symplectic  basis $\left\{b'_p,\ldots, b'_{0},a_0,\ldots,a_{p}\right\}$ the matrix representing $A$ is
\begin{equation*}
	\left(\begin{array}{cc}
		B&0\\
		0&J({\lambda},p+1)^\tau
	\end{array}\right)
	\end{equation*}
where
\begin{equation}\label{eq:J}
	J(\lambda,m)=
	\left(\begin{smallmatrix}
		\lambda & 1 &  &  &  &  & \\
		 & \lambda & 1 &  &  & \makebox(0,0)0  & \\
		 &  & \lambda & 1 &  &  & \\
		 &  &  & \ddots & \ddots &  & \\
		 &\makebox(0,0)0  &  &  & \lambda & 1 & \\
		 &  &   &  &  & \lambda & 1\\
		 &  &  &  &  &  & \lambda
	\end{smallmatrix}\right)
\end{equation}

 is the elementary $m\times m$ Jordan matrix associated to $\lambda$.
 Since the matrix is symplectic, $B$ is the transpose of the inverse of  $J({\lambda},p+1)^\tau$ by (\ref{mattriang}),  so $B= J({\lambda},p+1)^{-1}$.\\This is the normal form  for $A$ restricted to $E^v_\lambda\oplus E^w_{\frac{1}{\lambda}}$. \\
If $\lambda=re^{i\phi}\notin \R$ we consider the symplectic basis $\{b'_p,\ldots,b'_0,a_0,\ldots, a_{p}\}$ of $E^v_\lambda\oplus E^w_{\frac{1}{\lambda}}$ as above and the conjugate symplectic basis $\{\overline{b'_p},\ldots,\overline{b'_0},\overline{a_0},\ldots, \overline{a_p}\}$ of $E^{\overline{v}}_{\overline{\lambda}}\oplus E^{\overline{w}}_{\frac{1}{\overline\lambda}}.$ Writing $b'_j=\frac{1}{\sqrt{2}}(u_j+iv_j)$ and 
$a_j=\frac{1}{\sqrt{2}}(w_j-ix_j)$
for all $0\le j\le p$ with the vectors $u_j,v_j,w_j, x_j$ in the real vector space $V,$ we get a symplectic basis $\!\{u_p,v_p\ldots,u_{0}, v_{0},w_0,x_0\ldots,w_{p}, x_{p}\}$ of the real subspace of $V$ whose complexification is $E^v_\lambda\oplus E^w_{\frac{1}{\lambda}}\oplus E^{\overline{v}}_{\overline{\lambda}}\oplus E^{\overline{w}}_{\frac{1}{\overline\lambda}}$. In this basis, the matrix representing $A$ is
\begin{equation*}
	\left(\begin{array}{cc}
		J_\R\bigr(\overline{\lambda},2(p+1)\bigl)^{-1}&0\\
		0&J_\R\bigr(\overline{\lambda},2(p+1)\bigl)^\tau
	\end{array}\right)
\end{equation*}
where $J_\R(re^{i\phi},2m)$ is the $2m\times 2m$ matrix written in terms of $2\times 2$ matrices as
\begin{equation}\label{eq:JR}
J_\R(re^{i\phi},2m):=	\left(\begin{smallmatrix}
		R(re^{i\phi}) & \Id &  &  &  &  &\\
		 &R(re^{i\phi}) & \Id &  &  &\makebox(0,0)0  & \\
		 &  & R(re^{i\phi})& \Id &  &  & \\
		 &  &  & \ddots & \ddots &  & \\
		 &\makebox(0,0)0  &  &  & R(re^{i\phi}) &  \Id & \\
		 &  &  &  &  &R(re^{i\phi}) &  \Id \\
		 &  &  &  &  &  & R(re^{i\phi})
	\end{smallmatrix}\right)
\end{equation}
 with $R(re^{i\phi})=\left(\begin{smallmatrix}
		r\cos \phi&-r\sin \phi\\
		r\sin \phi&r\cos \phi
	\end{smallmatrix}\right)$.
By induction, we get
\begin{theorem}[Normal form for $A_{\vert V_{[\lambda]}}$ for $\lambda\notin S^1.$]\label{thm:normal1}
Let $\lambda\notin S^1$ be an eigenvalue of $A$. Denote $k:=\dim_\C \Ker (A-\lambda\Id)$ (on $V^\C$) and $p$  the smallest integer so that $(A-\lambda\Id)^{p+1}$ is identically zero on the generalized eigenspace $E_\lambda$.
\begin{itemize}
\item
If $\lambda\neq \pm1$ is a real eigenvalue of $A$,
there exists a symplectic basis of $V_{[\lambda]}$
in which the matrix representing  the restriction of $A$ to $V_{[\lambda]}$ is a symplectic direct sum of $k$
matrices of the form 
$$
        \left(\begin{array}{cc}
		J({\lambda},p_j+1)^{-1}&0\\
		0& J({\lambda},p_j+1)^{\tau}
	\end{array}\right)
$$
	 with $p=p_1\ge p_2\ge \dots \ge p_k$   and $J(\lambda,k)$ defined by (\ref{eq:J}). To eliminate the ambiguity in the choice of $\lambda$ in ${[\lambda]}=\{ \lambda,\lambda^{-1}\}$ we can consider the
 real eigenvalue  such that $\lambda>1$. The size of the blocks is determined knowing the dimension
$\dim\left( \Ker (A-\lambda\Id)^r\right)$ for each $r\ge 1$. 
\item
If $\lambda=re^{i\phi}\notin(S^1\cup \R)$ is a complex eigenvalue of $A$, there exists a symplectic basis of $V_{[\lambda]}$ in which the matrix representing the restriction of $A$ to $V_{[\lambda]}$ is a symplectic direct sum of $k$
matrices of the form
$$
	\left(\begin{array}{cc}
		J_\R\bigl(re^{-i\phi},2(p_j+1)\bigr)^{-1}&0\\
		0&J_\R\bigl(re^{-i\phi},2(p_j+1)\bigr)^\tau
	\end{array}\right)
$$
with $p=p_1\ge p_2\ge \dots \ge p_k$ and $J_\R(re^{i\phi},k)$ defined by (\ref{eq:JR}). To eliminate the ambiguity in the choice of $\lambda$ in ${[\lambda]}=\{ \lambda,\lambda^{-1}, \overline{\lambda},\overline{\lambda}^{-1}\}$ we can choose the 	 eigenvalue $\lambda$ with a positive imaginary part and a modulus greater than $1$. The size of the blocks is determined, knowing the dimension
$\dim_\C\left( \Ker (A-\lambda\Id)^r\right)$ for each $r\ge 1$. 
\end{itemize}
This normal form is unique, when  a choice of $\lambda$ in the set $[\lambda]$ is fixed.
\end{theorem}
\section{Normal forms  for  $A_{\vert V_{[\lambda]}}$  when $\lambda=\pm 1.$}

In this situation $[\lambda]=\{ \lambda\}$ and  $V_{[\lambda]}$ is  the generalized real eigenspace of eigenvalue $\lambda$, still denoted --with a slight abuse of notation-- $E_\lambda$.  Again, $p$ denotes the largest integer such that $(A-\lambda\id)^{p}$ does not vanish identically
on  $E_\lambda$.
We consider $\widetilde{Q}_p :\raisebox{.2ex}{$E_\lambda$}/\raisebox{-.2ex}{$\Ker(A-\lambda\Id)^p$} \times 	\raisebox{.2ex}{$E_\lambda$}/\raisebox{-.2ex}{$\Ker(A-\lambda\Id)^p$}\rightarrow \R$
the non degenerate form defined by $\widetilde{Q}_p \bigl([v],[w]\bigr)=\Omega \bigl( (A-\lambda\Id)^p v,w \bigr).$
We see directly from equation (\ref{eq:Qbiendef}) that $\widetilde{Q}_p$ is symmetric if $p$ is odd and antisymmetric if $p$ is even.\\
\subsection{ If $p=2k-1$ is odd}
	we choose $v\in E_\lambda$ such that
\begin{equation*}
	\widetilde{Q}\bigl([v],[v]\bigr)=\Omega \bigl( (A-\lambda\Id)^p v,v \bigr) \ne 0
\end{equation*}
and consider the smallest $A$-invariant subspace $E^v_\lambda$ of $E_\lambda$  containing $v$; it is  spanned by
\begin{equation*}
	\bigl\{ a_p:=(A-\lambda\Id)^pv,\ldots, a_i:=(A-\lambda\Id)^{i}v,\ldots, a_{0}:= v \bigr\}.
\end{equation*}
We have
\begin{itemize}
	\item[~] $\Omega(a_i,a_j)=0$ if $i+j\ge p+1(=2k)$ by equation (\ref{eq:zero});
	\item[~] $\Omega(a_i,a_{p-i})\ne0;$ by equation(\ref{eq:p}) and  by the choice of $v .$
\end{itemize}

\noindent Hence $E^v_\lambda$ is a symplectic subspace because,  in the  basis defined by the $e_i $'s, $\Omega$ has the triangular form
$\left(\begin{smallmatrix}
	0&&\overline{\ast}\\
	&\adots&\\
	\overline{\ast}&&\ast
\end{smallmatrix}\right)$
and has a non-zero determinant.\\
We can choose $v$ in $E_\lambda\subset V$ so that $\Omega \bigl((A-\lambda\Id)^{k}v,(A-\lambda\Id)^{k-1}v \bigr) = \lambda \s$ with $\s=\pm 1$ by rescaling the vector and 
one may further assume, by lemma \ref{lem:Todd}, that 
$$
T_{i,j}(v)=\frac{1}{\lambda^i}\frac{1}{\lambda^j}\Omega \bigl((A-\lambda\Id)^{i}v,(A-\lambda\Id)^{j}v \bigr) = 0\quad \textrm{for all }~0\le i,j\le k-1.
$$
We now construct a symplectic  basis $\left\{a'_p,\ldots, a'_{k},a_0,\ldots,a_{k-1}\right\}$ of $E^v_\lambda$, extending\\
$\left\{a_0,\ldots,a_{k-1}\right\}$, by a Gram-Schmidt procedure,
having chosen  $v$ as above.
We  define inductively   on $0\le j\le k-1$
\begin{itemize}
\item[~] $a'_{p}:= \frac{1}{\Omega(a_p,a_0)}a_p $;
\item[~] $a'_{p-j}=\frac{1}{\Omega(a_{p-j},a_j)} \bigl( a_{p-j}-\sum_{k<j}\Omega(a_{p-j}, a_k)a'_{p-k}\bigr),$\\
so that any $a'_j$ is a linear combination of the $a_r$'s with $r\ge j$ and in particular $a'_k=\frac{1}{\s\lambda}a_k+\sum_{j=1}^{k-1}c_ja_{k+j}$ .
\end{itemize}	
In the symplectic  basis $\left\{a'_p,\ldots, a'_{k},a_0,\ldots,a_{k-1}\right\}$ the matrix representing $A$ is
\begin{equation*}
A'=	\left(\begin{array}{cc}
		B&C\\
		0&J(\lambda,k)^\tau
	\end{array}\right)
	\end{equation*}
with $J(\lambda,m)$ defined by (\ref{eq:J}) and with $C$ identically zero except for the last column, and the coefficient $C^k_k=\s \lambda$.
 Since the matrix is symplectic, $B$ is the transpose of the inverse of  $J(\lambda,p+1)^\tau$ by (\ref{mattriang}),  so $B= J(\lambda,k)^{-1}$
 and $J(\lambda,k)C$ is symmetric with zeroes except in the last column, hence diagonal of the form $\operatorname{diag} \bigl(0,\ldots,0, \s\bigr)$.
Thus
\begin{equation*}
	\left(\begin{array}{cc}
		J(\lambda,k)^{-1}&{J(\lambda,k)^{-1}\textrm{diag} \bigl(0,\ldots,0, \s\bigr)}\\
		0&J(\lambda,k)^\tau
	\end{array}\right),
\end{equation*}
with $\s= \pm 1$, is the normal form of $A$ restricted to $E^v_\lambda$. Recall that
$$
\s=\lambda^{-1}\,\Omega \bigl((A-\lambda\Id)^{k}v,(A-\lambda\Id)^{k-1}v \bigr).
$$
\subsection{ If $p=2k$ is even}
we choose $v$ and $w$ in $E_\lambda$ such that
\begin{equation*}
	\widetilde{Q}\bigl([v],[w]\bigr)=\Omega\bigl((A-\lambda\Id)^p v,w\bigr)=\lambda^p=1
\end{equation*}
and we consider the smallest $A$-invariant subspace $E^v_\lambda\oplus E^w_\lambda$ of $E_\lambda$  containing $v$ and $w.$
It is of dimension $4k+2.$
Remark that $\Omega\bigl((A-\lambda\Id)^p v,v\bigr)=0$.  
We can choose $v$ so that
$$
T_{r,s}(v)=\frac{1}{\lambda^{r+s}}\Omega \bigl((A-\lambda\Id)^{r}v,(A-\lambda\Id)^{s}v \bigr) = 0\quad \textrm{for all }~r,s.
$$
Indeed, by formula \eqref{eq:Aij}  we have $T_{i,j}(v)=- T_{i+1,j}(v)-T_{i+1,j-1}(v)$.
Observe that $T_{i,j}(v)=-{T_{j,i}(v)}$ so that  $T_{i,i}(v)=0$ and $T_{j,i}(v)=- T_{j,i+1}(v)-T_{j-1,i+1}(v)$.
We   proceed by induction, as in lemma \ref{lem:Todd} : 
\begin{itemize}
\item
  $T_{p,0}(v)=0$ implies
$T_{p-r,r}(v)=0$ for all $0\le r\le p$ by equation \eqref{eq:p}.
\item We assume by decreasing induction on $J$, starting from $J=p$, that     we have 
$T_{{i},j}(v)=0$ for all $i+ j\ge J$. 
Then we have $T_{{J-1-s} ,s}(v)=- T_{{J-1-s} ,s+1}(v)-T_{{J-2-s} ,s+1}(v)$; the first term on the righthand side 
vanishes by the induction hypothesis,
so $T_{{J-1} ,0}(v)=(-1)^{s}T_{J-1-s,s}(v)=(-1)^{J-1}T_{0,J-1}(v)=(-1)^JT_{J-1,0}$.

If $T_{J-1,0}(v)=\alpha\neq 0$, $J$ must be even  and we replace $v$ by 
$$
	v'=v+\tfrac{\alpha}{2\lambda^{p-J+1}}(A-\lambda\Id)^{p-J+1}w.
$$
Then 
$v'\in E^v_\lambda\oplus E^w_\lambda , \,E^v_\lambda\oplus E^w_\lambda=E^{v'}_\lambda\oplus E^w_\lambda,\,  \Omega\bigl((A-\lambda\Id)^p v',w\bigr)=\lambda^p$ and 
$T_{i,j}(v')=T_{i,j}(v)=0$ for all $i+j\ge J$ but now
\begin{eqnarray*}
T_{J-1,0} (v')&=&T_{J-1,0} (v)+\tfrac{\alpha}{2\lambda^p}\Omega \bigl( (A-\lambda\Id)^{p}w,v\bigr)\\
&& \ \ \ \ +\tfrac{\alpha}{2\lambda^{p}}\Omega \bigl( (A-\lambda\Id)^{J-1}v,(A-\lambda\Id)^{p-J+1}w\bigr)\\
&& \ \ \ \ +\tfrac{\alpha^2}{4\lambda^p}\Omega \bigl( (A-\lambda\Id)^{p}w,(A-\lambda\Id)^{p-J+1}w\bigr)\\
&=& \alpha-\frac{\alpha}{2}-\frac{\alpha}{2}=0
\end{eqnarray*}
 so that $T_{i,j} (v')=0$ for all $i+j\ge J-1$ and the induction proceeds.
\end{itemize}
We assume from now on that we have chosen  $v$ and $w$ in $E_\lambda$ so that\\
$\Omega\bigl((A-\lambda\Id)^p v,w\bigr)=1$ and $\Omega\bigl((A-\lambda\Id)^r v,(A-\frac{1}{\lambda}\Id)^{s}v\bigr)=0$
for all $r,s$.\\ We can proceed similarly with $w$ so we can thus furthermore assume  that \\
$\Omega\Bigl((A-\lambda\Id)^j w,\bigl(A-{\lambda}\Id\bigr)^{k}w\Bigr)=0$ for all $j,k$.

A basis of $E^v_\lambda\oplus E^w_\lambda$ is given by
$$\bigl\{a_p=(A-\lambda\Id)^pv, \ldots,a_0=v, 
b_0=w, \ldots, b_{p}=(A-\lambda\Id)^pw\bigr\}.$$
We have
\begin{itemize}
	\item[~] $\Omega(a_i,a_j)=0$ and $\Omega(b_i,b_j)=0$ by the choice of $v$ and $w$;
	\item[~] $\Omega(a_i,b_j)=0$ if $i+j>p$  by equation (\ref{eq:zero}) ;
         \item[~] $\Omega(a_i,b_{p-i})=\neq 0$ by equation (\ref{eq:p}) and the choice of
          of $v,w .$
\end{itemize}
The  matrix representing $\Omega$ has the form
$\left(\begin{smallmatrix}
0&\vline&\begin{smallmatrix}
	\overline{\ast}&&0\\
	&\ddots&\\
	\ast&&\overline{\ast}\\
	~&~&~
\end{smallmatrix}\\
\hline
\begin{smallmatrix}
         ~&~&~\\
	\overline{\ast}&&\ast\\
	&\ddots&\\
	0&&\overline{\ast}
\end{smallmatrix}&\vline&0\\
\end{smallmatrix}\right)$
hence is  non singular and the subspace $E^v_\lambda\oplus E^w_\lambda$ is symplectic.
We now construct a symplectic  basis $\left\{a'_p,\ldots, a'_{0},b_0,\ldots,b_{p}\right\}$ of $E^v_\lambda\oplus E^w_{\frac{1}{\lambda}}$, 
extending  $\left\{b_0,\ldots,b_{p}\right\}$, using a Gram-Schmidt procedure on the $a_i$'s. 
We define inductively   on $j$
\begin{itemize}
\item[~] $a'_{p}:= \frac{1}{\Omega(a_p,b_0)}a_p $;
\item[~] $a'_{p-j}=\frac{1}{\Omega(a_{p-j},b_j)} \bigl( a_{p-j}-\sum_{k<j}\Omega(a_{p-j}, b_k)a'_{p-k}\bigr),$\\
so that any $a'_j$ is a linear combination of the $a'_k$ with $k\ge j$.
\end{itemize}	
In the symplectic  basis $\left\{a'_p,\ldots, a'_{0},b_0,\ldots,b_{p}\right\}$ the matrix representing $A$ is
\begin{equation*}
	\left(\begin{array}{cc}
		B&0\\
		0&J(\lambda,p+1)^\tau
	\end{array}\right) .
\end{equation*}
Hence,  the matrix 
$$
\left(\begin{array}{cc}
		J(\lambda,p+1)^{-1}&0\\
		0& J(\lambda,p+1)^{\tau}
	\end{array}\right)
$$
is a normal form for $A$ restricted to $E^v_\lambda\oplus E^w_\lambda$. Thus we have:

\begin{theorem}[Normal form for $A_{\vert V_{[\lambda]}}$ for $\lambda=\pm 1.$]\label{normalforms1}
Let $\lambda=\pm 1$ be an eigenvalue of $A$. 
There exists a symplectic basis of $V_{[\lambda]}$
in which the matrix representing the restriction of $A$ to $V_{[\lambda]}$ is a symplectic direct sum of 
matrices of the form 
$$
\left(\begin{array}{cc}
J(\lambda,r_j)^{-1}&C(r_j,\s_j,\lambda)\\
		0& J(\lambda,r_j)^{\tau}
		\end{array}\right)
$$
where $C(r_j,\s_j,\lambda):=J(\lambda,r_j)^{-1} \operatorname{diag}\bigl(0,\ldots,0, \s_j\bigr)$  with $\s_j\in \{0,1,-1\}$.
If $\s_j=0$,  then $r_j$ is odd.
The dimension of the eigenspace of eigenvalue $1$ is given by $2\card \{j \,\vert\, \s_j=0\}+\card\{j \,\vert \,\s_j\neq0\}$.
\end{theorem}
\begin{definition}
Given $\lambda \in \{\pm 1\}$,  we define, for any integer $k\ge 1$,  a bilinear form $\hat{Q}^\lambda_{2k}$ on 
$\Ker\left( (A-\lambda\id)^{2k}\right) $:
\begin{eqnarray}
\hat{Q}^\lambda_{2k} &:& \Ker\left( (A-\lambda\id)^{2k}\right)\times \Ker\left( (A-\lambda\id)^{2k}\right)\rightarrow \R\nonumber\\
&&(v,w)\mapsto \lambda\, \Omega\bigl((A-\lambda\id)^{k}v,(A-\lambda\id)^{k-1}w\bigr).
\end {eqnarray}
It is symmetric. 
\end{definition}

\begin{proposition}\label{sumd}
Given $\lambda \in \{\pm 1\}$, the number of positive (resp. negative) eigenvalues of the symmetric $2$-form $\hat{Q}^\lambda_{2k}$ is equal to the number
of $s_j$ equal to $+1$ (resp. $-1$) arising in blocks of dimension $2k$ (i.e. with corresponding $r_j=k$) in the normal decomposition of $A$ on $V_{[\lambda]}$ given in theorem \ref{normalforms1}.\\
On $V_{[\lambda]}$, we have:
\begin{equation}
\sum_j \s_j=\sum_{k=1}^{dim V}\operatorname{Signature}(\hat{Q}^\lambda_{2k})
\end{equation}
\end{proposition}

\begin{proof}
 On the intersection of  $ \Ker\left( (A-\lambda\id)^{2k}\right)$
with one of the symplectically orthogonal   subspaces  $E_\lambda^v$ 
constructed above for an odd  $p\ne 2k-1$, the form $\hat{Q}^\lambda_{2k}$ vanishes identically.
On the intersection of  $ \Ker\left( (A-\lambda\id)^{2k}\right)$ with a subspace $E_\lambda^v$ for a $v$ so that  $p=2k-1$ and $ \Omega \bigl((A-\lambda\Id)^{k}v,(A-\lambda\Id)^{k-1}v \bigr) =\lambda \s$ the only non vanishing component is  $\hat{Q}^\lambda_{2k}(v,v)=\s$. \\
 Indeed, $\Ker\left( (A-\lambda\id)^{2k}\right) \cap E_\lambda^v$ is spanned by $$\{ (A-\lambda\id)^{r} v\, ;\, r\ge 0 \textrm{ and } r+2k>p\, \},$$ and
$\Omega\bigl((A-\lambda\id)^{k+r}v,(A-\lambda\id)^{k-1+r'}v\bigr)=0$ when $2k+r+r'-1>p$  so the only non vanishing cases
arise when $r=r'=0$ and $p=2k-1$.\\
Similarly, the $2$ form  $\hat{Q}^\lambda_{2k}$ vanishes  on the intersection of $ \Ker\left( (A-\lambda\id)^{2k}\right)$ with   a subspace $E_\lambda^v\oplus E_\lambda^w$ constructed  above for an even $p$.
\end{proof}
The numbers $s_j$ appearing in the decomposition of $A$ are thus invariant of the matrix.
\begin{cor}\label{cor:un1}
The normal decomposition described in theorem \ref{normalforms1} is  determined by the eigenvalue $\lambda$, by
the dimension 
$\dim\bigl( \Ker (A-\lambda\Id)^r\bigr)$ for each $r\ge 1$, and by the rank and the signature of the
symmetric bilinear $2$-forms  $\hat{Q}^\lambda_{2k}$ for each $k\ge1$.  It is  unique up to a permutation of the blocks.\hfill{$\square$}
\end{cor}
\section{Normal forms  for  $A_{\vert V_{[\lambda]}}$  when  $\lambda=e^{i\phi} \in S^1\setminus \{ \pm 1\}.$}

We denote again by $p$  the largest integer such that $(A-\lambda\id)^{p}$ does not vanish identically
on  $E_\lambda$ and
we consider 
the non degenerate sesquilinear  form 
$$\widehat{Q} :\raisebox{.2ex}{$E_\lambda$}/\raisebox{-.2ex}{$\Ker(A-\lambda\Id)^p$}\times \raisebox{.2ex}{$E_{{\lambda}}$}/\raisebox{-.2ex}{$\Ker(A-\lambda\Id)^p$}\rightarrow \C$$
\begin{equation*}
	\widehat{Q} \bigl([v],[w]\bigr)=\overline{{\lambda}^p}\Omega \bigl( (A-\lambda\Id)^p v,\overline{w} \bigr).
\end{equation*}
Since $\widehat{Q}$ is non degenerate, we can choose $v\in E_\lambda$ such that $\widehat{Q}([v],[v])\ne 0$ thus $(A-\lambda\Id)^p v \ne 0$ and we consider the smallest $A$-invariant subspace, stable by complex conjugaison,
and containing $v$ :
$E^v_\lambda\oplus E^{\overline{v}}_{\overline{\lambda}} \subset E_\lambda\oplus E_{\overline{\lambda}}$.
A basis is given by
\begin{equation*}
	\bigl\{ a_{i}:=(A-\lambda\Id)^{i}v, b_{j}:=(A-\overline{\lambda}\Id)^{j}\overline{v} \quad 0\le i,j\le p \bigr\}.
\end{equation*}
We have $a_{i}=\overline{b_{i}}$ and
\begin{itemize}
	\item $\Omega(a_{i},a_{j})=0,~\Omega(b_{i}, b_{j})=0$ because $\Omega(E_\lambda,E_\lambda)=0;$
	\item $\Omega(a_{i},b_{k})=0$ if $i+k\ge p+1$ by equation (\ref{eq:zero});
	\item $\Omega(a_{i},b_{k})\ne 0$ if $p=i+k$ by equation (\ref{eq:p}) and by the choice of $v.$
\end{itemize}
\noindent We conclude that $E^v_\lambda\oplus E^{\overline{v}}_{\overline{\lambda}}$ is a symplectic subspace.\\

\subsection{If $ p=2k-1$ is odd}

observe  that $T_{k,k-1}(v):=\frac{1}{\lambda}\Omega \bigl((A-\lambda\Id)^{k}v,(A-{\overline{\lambda}}\Id)^{k-1}\overline{v} \bigr) = \s$ is real and can be put to $\pm 1$ by rescaling the vector (we could even put it to $1$ exchanging if needed $\lambda$ and its conjugate).
One may further assume, by lemma \ref{lem:Todd} that 
$$
T_{i,j}(v)=\frac{1}{\lambda^i}\frac{1}{{\overline{\lambda}}^j}\Omega \bigl((A-{{\lambda}}\Id)^{i}v,(A-{\overline{\lambda}}\Id)^{j}\overline{v} \bigr) = 0\quad \textrm{for all }~0\le i,j\le k-1.
$$
We consider the basis 
$\{a_{2k-1},\ldots,a_{k},b_p,\ldots,b_{k}, b_{0},\ldots b_{k-1}, a_{0},\ldots a_{k-1}\}$ for
such a vector $v$ with  $T_{k,k-1}(v)=\s=\pm1$ and $T_{i,j}(v)=0$ for all $0\le i,j\le k-1$;
the matrix representing $\Omega$ has the form
\begin{equation*}
	\makeatletter
		\setbox\strutbox\hbox{%
		\vrule\@height.5\baselineskip
		\@depth.2\baselineskip
		\@width\z@}
	\makeatother
	\left( \begin{smallmatrix}
		0&\vline&
		\begin{smallmatrix}
			\begin{smallmatrix}
				\overline{\ast}&&0\\
				&\ddots&\\
				\ast&&\overline{\ast}\strut
			\end{smallmatrix}
			&\vline &0         \\
			\hline
			0&\, \vline\,&
			\begin{smallmatrix}
				\overline{\ast}&&0\strut\\
				&\ddots&\\
				\ast&&\overline{\ast}\strut
			\end{smallmatrix} 
		\end{smallmatrix}\\
		\hline
		\begin{smallmatrix}
			\begin{smallmatrix}
				\overline{\ast}&&\ast\strut\\
				&\ddots&\\
				0&&\overline{\ast}\strut
			\end{smallmatrix}
			&\vline&0         \\
			\hline
			0& \vline&
			\begin{smallmatrix}
				\overline{\ast}&&\ast\strut\\
				&\ddots&\\
				0&&\overline{\ast}
			\end{smallmatrix} 
		\end{smallmatrix}&\,\vline\,&0\\
	\end{smallmatrix}\right)
\end{equation*}
 and we transform it by a Gram-Schmidt method into a symplectic basis composed of pairs of conjugate vectors, extending
$\{b_0,\ldots,b_{k-1}, a_0,\ldots,a_{k-1}\}$ on which $\Omega$ identically vanishes. 
We define
\begin{eqnarray*}
	a'_{2k-1} &=&\frac{1}{\Omega(a_{2k-1},b_0)} a_{2k-1},\\
	b'_{2k-1} &=&\frac{1}{\Omega(b_{2k-1},a_0)} b_{2k-1}=\overline{a'_{2k-1} }
\end{eqnarray*}
and,  inductively on increasing $j$ with $1<j\le k$
\begin{eqnarray*}
	a'_{2k-j} &=&\frac{1}{\Omega(a_{2k-j},b_{j-1})}\left(a_{2k-j}-\sum_{r=1}^{j-1}\Omega(a_{2k-j},b_{r-1})\, a'_{2k-r}\right),\\
	b'_{2k-j} &=&\overline{a'_{2k-j} }.
\end{eqnarray*}
Any $a'_{2k-j}$ is a linear combination of the $a_{2k-i}$ for $1\le i\le j$; reciprocally any
$a_{2k-j}$ can be written as a linear combination of the $a'_{2k-i}$ for $1\le i\le j$, and   the coefficient of $a'_{2k-j}$ is
equal to $\Omega(a_{2k-j},b_{j-1})$.\\
The basis
$\{a'_{2k-1},\ldots,a'_{k},b'_{2k-1},\ldots,b'_{k}, b_0,\ldots,b_{k-1}, a_0,\ldots,a_{k-1}\}$ is symplectic, and in that basis,
since $A(a_r)={\lambda} a_r+a_{r+1}$ and 
$A(b_r)=\overline{\lambda} b_r+b_{r+1}$ for all $r<2k-2$, 
the matrix representing $A$ is of the block upper triangular form
\begin{equation*}
	\left(\begin{array}{cccc}
		\ast & 0 & 0 & C\\
		 & \ast & \overline{C} & 0\\
		 &  &J(\overline{\lambda},k)^\tau& 0\\
		{\text{\large {0}}}&  &  &J({\lambda},k)^\tau
	\end{array}\right)
\end{equation*}
where $C$ is a $k\times k$ matrix such that the only non vanishing terms are on the last column ($C^i_{\,j}=0$ when $j<k$)
and $C^k_k=\Omega(a_k,b_{k-1})=\s{\lambda}$. The fact that the matrix is symplectic implies
that $S:=J(\overline{\lambda},k)C$ is hermitean; since $S^i_{\,j}=0$ when $j\neq k$, we have, \[
	C=J(\overline{\lambda},k )^{-1} \left(\begin{smallmatrix}0&\ldots&0&0\\
	                                                 \vdots&\ddots&\vdots&\vdots\\
	                                                 0&\ldots&0&0\\
						   0&\ldots& 0&\s
						\end{smallmatrix}\right)                     = C(k,\s,\overline{\lambda})
\]
and the matrix of the restriction of $A$ to the subspace $E^v_\lambda\oplus E^{\overline{v}}_{\overline{\lambda}}$ has the 
block triangular normal form
\begin{equation}
\left(\begin{array}{cccc}
		 J(\overline{\lambda},k )^{-1} & 0 & 0 & C(k,\s,\overline{\lambda})\\
		 &  J({\lambda},k )^{-1} & C(k,\s,{\lambda}) & 0\\
		 &  & J(\overline{\lambda},k )^{\tau} & 0\\
		{\text{\large{0}}} &  &  & J({\lambda},k )^\tau	\end{array}\right).
\end{equation}
Writing $a'_{2k-j}=\frac{1}{\sqrt 2}(e_{2j-1}-ie_{2j}),~b'_{2k-j}=\overline{a'_{2k-j}}=\frac{1}{\sqrt 2}(e_{2j-1}+ie_{2j})$, as well as 
$a_{j-1}=\frac{1}{\sqrt 2}(f_{2j-1}-i f_{2j})$ and $b_{j-1}=\overline{a_{j-1}}=\frac{1}{\sqrt 2}(f_{2j-1}+i f_{2j})$ for $1\le j\le k$,
the vectors $e_i,f_j$ all belong to the real subspace denoted $V^v_{[ \lambda]}$ of $V$ whose complexification is $E^v_\lambda\oplus  E^{\overline{v}}_{\overline{\lambda}}$ and we get a symplectic basis $$\{e_1,\ldots,e_{2k},f_1,\ldots,f_{2k} \}$$ of this real subspace $V^v_{[ \lambda]}$.  The matrix representing $A$ in this basis  is :  
\begin{equation}
	\left(\begin{array}{cc}
		\bigl(J_\R(\overline{\lambda},2k)\bigr)^{-1} & C_\R(k,\s,\overline{\lambda})\\
		0 & \bigl(J_\R(\overline{\lambda},2k)\bigr)^\tau
	\end{array}\right)
\end{equation}
where $J_\R(e^{i\phi},2k)$ is defined as in \eqref{eq:JR}
and where $C_\R(k,\s,e^{i\phi})$ is the $(p+1)\times (p+1)$ matrix written in terms of two by two matrices  as
\begin{equation}\label{eq:CR}
C_\R(k,\s,e^{i\phi})^\tau=\s\left(\begin{smallmatrix}
		0&\ldots &0&0\\
		\vdots&&\vdots&\vdots\\
		0&\ldots&0 &0\\
		(-1)^{k-1} R(e^{ik\phi})&\ldots &- R(e^{i2\phi})&R(e^{i\phi})
	\end{smallmatrix}\right)
\end{equation}
 with $R(e^{i\phi})=\left(\begin{array}{cc}
		\cos \phi&-\sin \phi\\
		\sin \phi&\cos \phi
	\end{array}\right)$ as before and $\s=\pm 1$.
This is the normal form of $A$ restricted to ${V^v_{[ \lambda]}}$; recall that
$$
\s={\lambda}^{-1}\, \Omega \bigl((A-\lambda\Id)^{k}v,(A-{\overline{\lambda}}\Id)^{k-1}\overline{v} \bigr).
$$

\subsection {If $p=2k$ is even}

we observe that $\Omega \bigl( (A-\overline{\lambda}\Id)^{k}\overline{v},(A-\lambda\Id)^{k}v \bigr)$ is purely imaginary and  we choose $v$ so that it is  $\Omega \bigl( (A-\overline{\lambda}\Id)^{k}\overline{v},(A-\lambda\Id)^{k}v \bigr) =\s i $  where $\s=\pm 1$ (remark that the sign changes if one  permutes  $\lambda$ and $\overline{\lambda}$).
We can further choose the vector  $v$ so that  :
\begin{eqnarray}
\Omega \left((A-\lambda\Id)^kv,(A-\overline{\lambda}\Id)^{k-1}\overline{v}\right)&=&\half \lambda \s i \label{condv}\\ 
T_{i,j}(v):=\frac{1}{\lambda^i {\overline{\lambda}}^j}\Omega \left( (A-\lambda\Id)^iv,(A-{\overline{\lambda}}\Id)^{j}{\overline{v}}\right) &=& 0   \qquad \quad  \textrm{ for all} 0\le i,j\le k-1 \nonumber; 
\end{eqnarray}
Indeed, as before, by 	 \eqref{eq:Aij},  we have $T_{i,j}(v)=- T_{i+1,j}(v)-T_{i+1,j-1}(v)$ and  $T_{i,j}(v)=-{\overline{T_{j,i}(v)}}$ and we proceed   as in lemma \ref{lem:Todd} by decreasing induction on $i+j$:
\begin{itemize}
\item if $T_{{k},{k-1}}(v)= \alpha_1$, since $T_{k-1,{k}}(v)=\s i -T_{{k},{k-1}}(v)$ the imaginary part of $\alpha_1$ is equal to $\half \s i$ and  we replace $v$ by $v-\frac{\alpha_1}{2\lambda \s i}(A-\lambda\Id)v$; it generates the same $A$-invariant subspace
and the quantities $T_{i,j}(v)$ do not vary for $i+j\ge 2k$ but now $T_{{k},{k-1}}(v)= \alpha_1
-\frac{\alpha_1}{2\s i}T_{k+1,k-1}(v)+\frac{\overline{\alpha_1}}{2\s i}T_{k,k}(v)=\alpha_1-\half \alpha_1-\half{\overline{\alpha_1}}=\half \s i$ since $ T_{k,k}(v)=-T_{k+1,k-1}(v)=-\s i$; so we can now assume $T_{{k},{k-1}}(v)=\half  \s i$;
\item if $T_{{k-1},{k-1}}(v)= \alpha_2$, this $\alpha_2$ is purely imaginary and we replace $v$ by $v-\frac{\alpha_2}{2\lambda^2\s i}(A-\lambda\Id)^{2}v$; it generates the same $A$-invariant subspace
and the quantities $T_{i,j}(v)$ do not vary for $i+j\ge 2k-1$;
now $T_{{k-1},{k-1}}(v)=\alpha_2-\frac{\alpha_2}{2\s i}T_{{k+1},{k-1}}(v)+\frac{\overline{\alpha_2}}{2\s i}T_{{k-1},{k+1}}(v)=\alpha_2-\half \alpha_2+\half {\overline{\alpha_2}}= 0$. We may thus assume this property to hold for $v$.
\item if $T_{{k-2},{k-1}}(v)= \alpha_3=-T_{{k-1},{k-1}}(v)-T_{{k-1},{k-2}}(v)={\overline{T_{{k-2},{k-1}}(v)}}$, this $\alpha_3$ is real and we replace $v$ by $v-\frac{\alpha_3}{2\lambda^3\s i}(A-\lambda\Id)^{3}v$; it generates and the the same $A$-invariant subspace and the quantities $T_{i,j}(v)$ do not vary for $i+j\ge 2k-2$;
now $T_{{k-2},{k-1}}(v)=\alpha_3-\frac{\alpha_3}{2\s i}T_{{k+1},{k-1}}(v)+\frac{\overline{\alpha_3}}{2\s i}T_{{k-2},{k+2}}(v)=0$, since
$T_{{k+1},{k-1}}(v)=-T_{{k},{k}}(v)=-T_{{k-2},{k+2}}(v)=\s i$; hence also 
$T_{{k-1},{k-2}}(v)= 0$; 
\item we now assume by induction to have a $J>1$ so  that $T_{i,j}(v)=0$ for all $0\le i,j\le k-1$ so that $i+j> 2k-1-J$;
\item if $T_{{k-J},{k-1}}(v)= \alpha_{J+1}$, then $T_{{k-J},{k-1}}(v)=(-1)^{J-1}T_{{k-1},{k-J}}(v)$
so that $\alpha_{J+1}$ is real when $J$ is even and is imaginary when $J$ is odd; we replace $v$ by $v-\frac{\alpha_{J+1}}{2\lambda ^{J+1}\s i}(A-\lambda\Id)^{J+1}v$; it sgenerates  the same $A$-invariant subspace
and the quantities $T_{i,j}(v)$ do not vary for $i+j\ge 2k-J$, but
 now $T_{{k-J},{k-1}}(v)= \alpha_{J+1} -\frac{\alpha_{J+1}}{2\s i}T_{{k+1},{k-1}}(v)+\frac{\overline{\alpha_{J+1}}}{2\s i}T_{{k-J},{k+J}}(v)=\alpha_{J+1}-\frac{\alpha_{J+1}}{2}+(-1)^{J+1}\frac{\overline{\alpha_{J+1}}}{2}=0$.\\
 Hence also 
$T_{{k-J+1},{k-2}}(v)= 0,\ldots ~T_{{k-1},{k-J+1}}(v)= 0$; so the induction step is proven.
\end{itemize}
{\small{\begin{remark}
For such a  $v$, all $T_{i,j}(v)$ are determined inductively and we have
\begin{eqnarray*}
T_{i,j}(v)&=&0 ~~~~~\textrm{if~} i+j\ge 2k+1\quad \textrm{and}\quad  \textrm{ for all } 0\le i,j \le k-1\\
T_{k-r,k+r}(v)&=&(-1)^{r+1} \s i ~~~~~ \textrm{ for all  }~ 0\le r\le k\\
T_{k-r,k+m}(v)&=&(-1)^{r+1}\frac{\s i}{2}\frac{(r+m)(r-1)!}{m!(r-m)! }~~~ \textrm{ for all  }  0\le m\le  r\le k,\, r>1\\
T_{i,j}(v)&=&T_{j,i}(v)~~~~ \textrm{ for all  }  i,j.
\end{eqnarray*}
\end{remark}}}
With the notation $a_i=(A-\lambda\Id)^{i}v, b_i=(A-\overline{\lambda}\Id)^{i}\overline{v}$, we consider the basis 
$$\{a_{2k},\ldots, a_{k+1}, b_{2k},   \ldots,   
b_{k+1},b_k;
b_0,\ldots, b_{k-1}, a_0, \ldots ,  a_{k-1},  a_k\}
 $$ 
for such a vector $v$;
the matrix representing $\Omega$ in this basis has the form

\begin{equation*}\makeatletter
       \setbox\strutbox\hbox{%
         \vrule\@height.5\baselineskip
               \@depth.2\baselineskip
               \@width\z@}\makeatother
\left(
   \begin{smallmatrix} 0     & \vline & 0 & \vline & 0   & \vline & 
     \begin{smallmatrix}
       \overline{\ast}\strut       &        & 0                             \\
                             & \ddots &                               \\
       \ast \strut                 &        & \overline{\ast}\strut
     \end{smallmatrix}
                             & \vline & 0 & \vline & 0                \\ 
     \hline 0                & \vline & 0 & \vline & 0   & \vline & 0 & \vline & 
     \begin{smallmatrix}
       \overline{\ast}\strut &        & 0                             \\
                             & \ddots &                               \\
       \ast \strut                 &        & \overline{\ast}\strut
     \end{smallmatrix}
                             & \vline & 0                             \\ 
     \hline 0                & \vline & 0 & \vline & 0   & \vline & 0 & \vline & 
     \begin{smallmatrix}
       \ast \strut                 & \ldots & \ast\strut
     \end{smallmatrix}
                             & \vline & \s i                            \\ 
     \hline
     \begin{smallmatrix}
       \overline{\ast}\strut       &        & \ast \strut                         \\
                             & \ddots &                               \\
       0                     &        & \overline{\ast}\strut
     \end{smallmatrix}
                             & \vline & 0 & \vline & 0   & \vline & 0 & \vline & 0 & \vline & 
     \begin{smallmatrix}
       \ast \strut                                                          \\
       \vdots                                                         \\
       \ast\strut
     \end{smallmatrix}                                                \\ 
     \hline 0                & \vline & 
     \begin{smallmatrix}
       \overline{\ast}\strut       &        & \ast \strut                         \\
                             & \ddots &                               \\
       0                     &        & \overline{\ast}\strut
     \end{smallmatrix}
                             & \vline & 
     \begin{smallmatrix}
       \ast \strut                                                          \\
       \vdots                                                         \\
       \ast\strut
     \end{smallmatrix}
                             & \vline & 0 & \vline & 0   & \vline & 0 \\ 
     \hline 0                & \vline & 0 & \vline & -\s i\strut & \vline & 
     \begin{smallmatrix}
       \ast \strut                 & \cdots & \ast\strut
     \end{smallmatrix}
                             & \vline & 0 & \vline & 0                \\
   \end{smallmatrix}
\right).
\end{equation*}

 We transform (by a Gram-Schmidt method) the basis  above into a symplectic basis,  composed of pairs of conjugate vectors
 (up to a factor) and 
 extending 
$$
b_0,\ldots, b_{k-1}, a_0,   \ldots, a_{k-1}
$$
 on which $\Omega$ identically vanishes.  We define inductively, for increasing $j$ with $1\le j\le k-1$
 \begin{eqnarray*}
a'_{2k}:&=&\frac{1}{\Omega\bigl((A-{\lambda}\id)^{2k}v,\overline{v}\bigr)} (A-\lambda\id)^{2k}v=\frac{1}{\Omega(a_{2k},b_{0})}a_{2k} \\
b'_{2k}:&=&\frac{1}{\Omega\bigl((A-\overline{\lambda}\id)^{2k},\overline{v},v\bigr)} (A-\overline{\lambda}\id)^{2k}\overline{v}=\frac{1}{\Omega(b_{2k},a_{0})}b_{2k}=\overline{a'_{2k}}\\
a'_{2k-j}&=&\frac{1}{\Omega(a_{2k-j},b_j)} \left( a_{2k-j}-\sum_{r=0}^{j-1} \Omega(a_{2k-j},b_r) a'_{2k-r} \right) \\
b'_{2k-j}&=&\frac{1}{\Omega(b_{2k-j},a_j)} \left( b_{2k-j}-\sum_{r=0}^{j-1} \Omega(b_{2k-j},a_r) b'_{2k-r}\right)=\overline{a'_{2k-j}} \\
a'_{k}&=& a_{k} -\sum_{r=0}^{k-1}\Omega(a_{k},b_r)a'_{2k-r}  \\
b'_{k}&=&\frac{1}{\Omega(b_k,a_k)} \left( b_{k}-\sum_{r=0}^{k-1}\Omega(b_{k},a_r)b'_{2k-r} \right)=\frac{1}{i\s}\overline{a'_k}.\\
\end{eqnarray*}
Each  $a'_{2k-j}$ is a linear combination of the $(A-\lambda\id)^{2k-r}v$ for $0\le r\le j$. 
The basis 
$$
\{ a'_{2k},\ldots, a'_{k+1},b'_{2k},\ldots,b'_{k+1},b'_k; b_0,\ldots,b_{k-1},a_0,\ldots, a_{k-1},a'_k \}
$$
is now symplectic.
Since  $A(a_r)={\lambda} a_r+a_{r+1}$ for all $r<2k$, and $A (a_{2k})=\lambda a_{2k}$,
the matrix representing $A$  in that basis is of the form
\begin{equation*}
	\left(
		\begin{smallmatrix}
			\begin{smallmatrix}
				~&A_1&~
			\end{smallmatrix}
				& 0 & 0 &
				\left(
					\begin{smallmatrix}
						& {0}&
							{\begin{smallmatrix}
								c^{2k}& d^{2k}\\     
								\vdots&\vdots \\
								c^{k+1}&d^{k+1}\\
							\end{smallmatrix}} \\
					\end{smallmatrix}
				\right)\\[2mm]
			0 &
				\begin{smallmatrix}
					~&A_2&~
				\end{smallmatrix}
				&
				\left(
					\begin{smallmatrix}
						&{{0}}&
							{{\begin{smallmatrix}
								e^{2k}\\             
								\vdots \\
								e^{k+1} \\
								e^{k}
							\end{smallmatrix}}}\\
					\end{smallmatrix}
				\right)
				& 0\\[3mm]
			0 & 0 & J(\overline{\lambda},k)^\tau &0\\[3mm]
			0 & 0 & 0 &J({\lambda},k+1)^\tau\\[3mm]
		\end{smallmatrix}
	\right)
\end{equation*}
with $A(b_{k-1})=\overline\lambda b_{k-1} +\sum_{j=0}^k e^{k+j} b'_{k+j}$, 
$A(a_{k-1})=\lambda a_{k-1} +a'_k+ \sum_{j=1}^k c^{k+j} a'_{k+j}$ and 
$A(a'_{k})=\lambda a'_{k} +\sum_{j=1}^k d^{k+j} a'_{k+j}$.\\

Since a matrix $\left(\begin{array}{cc} A' & E\\ 0& D\end{array}\right)$ is symplectic  if and only if  $A'=(D^\tau)^{-1}$
and $D^\tau\, E$ is symmetric, we have
$$
A_1= J(\overline{\lambda},k)^{-1}    \qquad\qquad A_2=J({\lambda},k+1)^{-1}
$$
and 
 $$
J(\overline\lambda,k)\left( \begin{smallmatrix}& {{0}}& {{\begin{smallmatrix}c^{2k}& d^{2k}\\     
		                                                                                                                                         \vdots&\vdots \\ c^{k+1}&d^{k+1}\\  \end{smallmatrix}}} \\
		                      \end{smallmatrix}\right)
	=\left( J({\lambda},k+1)\left( \begin{smallmatrix}&{{0}}& {{\begin{smallmatrix}e^{2k}\\             
		                                                                                                                 \vdots \\ e^{k+1} \\e^{k}  \end{smallmatrix}}}\\ \end{smallmatrix}\right)\right)^\tau.
$$ This implies 
$$
J(\overline{\lambda},k)\left( {{\begin{smallmatrix}c^{2k}& d^{2k}\\  \vdots&\vdots \\ c^{k+2}&d^{k+2}\\c^{k+1}&d^{k+1}\\  \end{smallmatrix}}}\right)=\left( {{\begin{smallmatrix}0& 0\\  \vdots&\vdots \\0&0\\ s_1&s_2\\  \end{smallmatrix}}}\right)\qquad\qquad 
J({\lambda},k+1)\left({{\begin{smallmatrix}e^{2k}\\  \vdots \\e^{k+2}\\ e^{k+1} \\e^{k}  \end{smallmatrix}}}\right)=\left(  {{\begin{smallmatrix}0\\  \vdots \\0\\ s_1 \\ s_2  \end{smallmatrix}}}\right)
$$ 
so that $s_1=\overline\lambda c^{k+1}$ and $s_2=\overline\lambda d^{k+1}$. Now 
\begin{eqnarray*}
A(a'_{k})&=&A\Bigl(a_k+\sum_{j\ge 1}F^j_ka_{k+j}\Bigr)=\lambda a'_{k}+a_{k+1}+\sum_{j\ge 1}F^j_ka_{k+j+1}\\
&=&\lambda a'_{k}+a'_{k+1}\Omega(a_{k+1},b_{k-1})+\sum_{j\ge 1}F^{'j}_ka'_{k+j+1}
\end{eqnarray*}
so that $d^{k+1}=\Omega(a_{k+1},b_{k-1})=\lambda^2 i\s$ and $s_2=\lambda i\s$. We also have
$$A(a_{k-1})=\lambda a_{k-1}+a_k=\lambda a_{k-1}+a'_k+\Omega(a_k,b_{k-1})a'_{k+1}+\sum_{j\ge2} G^j a'_{k+j}$$
so that $c^{k+1}=\Omega(a_k,b_{k-1})=\lambda\half i\s$ and $s_1=\half i\s $.

We have thus shown that the matrix representing $A$ in the chosen basis has the block upper-triangular normal form 
\begin{equation}\label{normalformcircleodd}
\begin{pmatrix}
J(\overline\lambda,k)^{-1}&0&0&J(\overline\lambda,k)^{-1} S\\ 
&J(\lambda,k+1)^{-1}&J(\lambda,k+1)^{-1} S^\tau&0\\ 
&& J(\overline\lambda,k)^{\tau}&0\\ 
  \multicolumn{2}{c}{\text{\kern-0.7em\smash{\raisebox{0.75ex}{\large 0}}}} && J(\lambda,k+1)^{\tau}\\ \end{pmatrix} 
\end{equation}
where $S$ is the $k\times(k+1)$ matrix defined by
\begin{equation}\label{normalformcircleodd2}
S= S(k,d,\lambda):=\left(\begin{array}{ccccc} 
0&\ldots& 0&0&0\\
\vdots &&\vdots&\vdots&\vdots \\
0&\ldots& 0&0&0\\
0&\ldots &0&\half i\s&\lambda i\s\\
 \end{array}\right) .
\end{equation}

We write $a'_{2k+1-j}=\frac{1}{\sqrt 2}(e_{2j-1}-ie_{2j}),~b'_{2k+1-j}=\overline{a'_{2k+1-j}}=\frac{1}{\sqrt 2}(e_{2j-1}+ie_{2j})$, as well as 
$a_{j-1}=\frac{1}{\sqrt 2}(f_{2j-1}-i f_{2j})$ and $b_{j-1}=\overline{a_{j-1}}=\frac{1}{\sqrt 2}(f_{2j-1}+i f_{2j})$ for $1\le j\le k$,
and $a'_k=\frac{1}{\sqrt 2}(e_{2k+1}+id\, f_{2k+1}), b'_k=-id\overline{a'_{k}}=\frac{1}{\sqrt 2}(-f_{2k+1}-id\, e_{2k+1})$.
The vectors $e_i,f_j$ all belong to the real subspace $V^v_{[ \lambda]}$ of $V$ whose complexification is $E^v_\lambda\oplus  E^{\overline{v}}_{\overline{\lambda}}$ and we get a symplectic basis $$\{e_1,\ldots,e_{2k+1},f_1,\ldots,f_{2k+1} \}$$ of $V^v_{[ \lambda]}$. In this basis, the matrix representing $A$ is :  

\begin{equation*}
	\makeatletter
		\setbox\strutbox\hbox{%
		\vrule\@height.5\baselineskip
		\@depth.2\baselineskip
		\@width\z@}
	\makeatother
	\left(\begin{smallmatrix}
		\bigl(J_\R(\overline{\lambda},2k)\bigr)^{-1}&\vline&\s\, U^2(\phi)&\vline&
			\begin{smallmatrix}
				0\strut\\
				\vdots\\ 
				0\strut
			\end{smallmatrix}
			&
			\begin{smallmatrix}
				\cdots\\
				\phantom{\vdots}\\
				\cdots
			\end{smallmatrix}&
			\begin{smallmatrix}
				0\strut\\
				\vdots\\
				0\strut
			\end{smallmatrix}
			&\frac{\s}{2} V^2(\phi) &\frac{-\s}{2}V^1(\phi)&\vline &U^1(\phi)\\
		\hline
		0\strut&\vline&\cos\phi&\vline&
			0&\ldots & 0& 1 & 0
		&\vline&\s \sin\phi\\
		\hline
		0
			&\vline&\begin{smallmatrix}
			0\strut\\
			\vdots\\
			0\strut\\
		\end{smallmatrix}
		&\vline& &&&\bigl(J_\R(\overline{\lambda},2k)\bigr)^\tau &&\vline&
		\begin{smallmatrix}
			0\strut\\
			\vdots\\
			0\\
		 \end{smallmatrix} \\
		 \hline
		0\strut&\vline&-\s\sin\phi&\vline&
		0&\ldots & 0&0  &-\s&\vline
		&\cos\phi\\
	\end{smallmatrix}\right)
\end{equation*}
where  $s=\pm 1$, $U^1(\phi), U^2(\phi),V^1(\phi)$ and $V^2(\phi)$ are real  $2k\times 1$ column matrices such that 		 
\begin{equation*}
  \left(V^1 (\phi) \, V^2(\phi) \right)= \left(\begin{matrix} (-1)^{k-1}  R(e^{i k\phi})\\
		 \vdots\\R(e^{i\phi})
		 \end{matrix}\right)
\end{equation*}
\begin{equation*}
 \left(U^1 (\phi) \, U^2(\phi) \right)= \left(\begin{matrix} (-1)^{k-1}  R(e^{i (k+1)\phi})\\
		 \vdots\\R(e^{i2\phi})\\
		 \end{matrix}\right)= \left(V^1 (\phi) \, V^2(\phi) \right)\left(R(e^{i\phi})\right).
\end{equation*}
This is the normal form of $A$ restricted to $V^v_{[ \lambda]}$. Recall that
$$
\s=i\Omega \bigl( (A-\lambda\Id)^{k}v,(A-\overline{\lambda}\Id)^{k}\overline{v} \bigr).  
$$
\begin{theorem}[Normal form for $A_{\vert V_{[\lambda]}}$ for $\lambda\in S^1\setminus \{\pm 1\}.$]\label{normalforms3}
Let $\lambda \in S^1\setminus \{\pm 1\}$ be an eigenvalue of $A$. 
There exists a symplectic basis of $V_{[\lambda]}$
in which the matrix representing the restriction of $A$ to $V_{[\lambda]}$ is a symplectic direct sum of 
 $4k_j\times 4k_j$ matrices ($k_j\ge 1$) of the form 
	\begin{equation}
		\makeatletter
			\setbox\strutbox\hbox{%
			\vrule\@height.5\baselineskip
			\@depth.2\baselineskip
			\@width\z@}
		\makeatother
		\left(\begin{smallmatrix}
			\bigl(J_\R(\overline{\lambda},2k_j)\bigr)^{-1}&\vline&
				\begin{smallmatrix} 
					\begin{smallmatrix}
						0\\
						\vdots\\
						0\strut \\
					\end{smallmatrix}
					&
					\begin{smallmatrix}
						\cdots\\
						\phantom{\vdots}\\
						\cdots\strut
					\end{smallmatrix}&
					\begin{smallmatrix}
						0\\
						\vdots\\
						0\strut \\
					\end{smallmatrix}
					&\s_j\, V_{k_j}^1(\phi) &\s_j\, V_{k_j}^2(\phi)\\
				\end{smallmatrix})\\
			\hline
			0\strut &\vline&\strut \bigl(J_\R(\overline{\lambda},2k_j)\bigr)^\tau\strut\\
		\end{smallmatrix}\right)
	\end{equation}

and    $(4k_j+2)\times (4k_j+2)$ matrices ($k_j\ge 0$) of the form
\begin{equation}\label{eqref:3-1}
	\makeatletter
		\setbox\strutbox\hbox{%
		\vrule\@height.5\baselineskip
		\@depth.2\baselineskip
		\@width\z@}
	\makeatother
	\left(\begin{smallmatrix}
		\bigl(J_\R(\overline{\lambda},2k_j)\bigr)^{-1}&\vline&\s_j\, U_{k_j}^2(\phi)&\vline&
			\begin{smallmatrix}
				0\strut\\
				\vdots\\ 
				0\strut
			\end{smallmatrix}
			&
			\begin{smallmatrix}
				\cdots\\
				\phantom{\vdots}\\
				\cdots
			\end{smallmatrix}&
			\begin{smallmatrix}
				0\\
				\vdots\\
				0
			\end{smallmatrix}
			&\frac{\s_j}{2} V_{k_j}^2(\phi) &\frac{-\s_j}{2}V_{k_j}^1(\phi)&\vline &U_{k_j}^1(\phi)\\
		\hline
		0\strut&\vline&\cos\phi&\vline&
			0&\ldots & 0& 1 & 0
		&\vline&\s_j\sin\phi\\
		\hline
		0
		&\vline&\begin{smallmatrix}
			0\strut\\
			\vdots\\
			0\strut\\
		\end{smallmatrix}
		&\vline& &&&\bigl(J_\R(\overline{\lambda},2k_j)\bigr)^\tau &&\vline&
		\begin{smallmatrix}
			0\strut\\
			\vdots\\
			0\\
		 \end{smallmatrix} \\
		 \hline
		0\strut&\vline&-\s_j\sin\phi&\vline&
		0&\ldots & 0&0  &-\s_j&\vline
		&\cos\phi\\
	\end{smallmatrix}\right)
\end{equation}
where $J_\R(e^{i\phi},2k)$ is defined as in \eqref{eq:JR},
where $  \left(V_{k_j}^1 (\phi) \, V_{k_j}^2(\phi) \right)$ is the $2k_j\times 2$ matrix
 defined by 
  \begin{equation}
  \left(V_{k_j}^1 (\phi) \, V_{k_j}^2(\phi) \right)=
 \left(\begin{matrix} (-1)^{k_j-1}  R(e^{i k_j\phi})\\
		 \vdots\\R(e^{i\phi})
		 \end{matrix}\right)
\end{equation}
with $R(e^{i\phi})=\left(\begin{array}{cc}
		\cos \phi&-\sin \phi\\
		\sin \phi&\cos \phi
	\end{array}\right)$, where  
	\begin{equation}\label{U^1_{k_j}U^2_{k_j}}
 \left(U^1_{k_j} (\phi) \, U^2_{k_j}(\phi) \right)= \left(V^1_{k_j} (\phi) \, V^2_{k_j}(\phi) \right)\left(R(e^{i\phi})\right)
\end{equation}
 and  where $\s_j=\pm 1$.
The complex dimension of the eigenspace of eigenvalue $\lambda$ in $V^\C$ is given by the number of such matrices. 
\end{theorem}
\begin{definition}
Given  $\lambda \in S^1\setminus \{\pm 1\}$,  we define, for any integer $m\ge 1$, a Hermitian form $\hat{Q}^\lambda_{m} $ on $\Ker\left( (A-\lambda\id)^{m}\right)$ by:
\[\begin{array}{lll}
\hat{Q}^\lambda_{m} :& \Ker\bigl( (A-\lambda\id)^{m}\bigr)\times \Ker\bigl( (A-\lambda\id)^{m}\bigr)\rightarrow \C\nonumber&\\
&\quad (v,w)\mapsto\frac{1}{ \lambda} \Omega\bigl((A-\lambda\id)^{k}v,(A-\overline{\lambda}\id)^{k-1}\overline{w}\bigr)  &\textrm { if }m=2k\\
&\quad (v,w)\mapsto  i\, \Omega\bigl((A-\lambda\id)^{k}v,(A-\overline{\lambda}\id)^{k}\overline{w}\bigr)  &\textrm { if }m=2k+1.
\end {array}\]
\end{definition}

\begin{proposition}\label{sumds}
For $\lambda \in S^1\setminus \{\pm 1\}$, the number of positive (resp. negative) eigenvalues of the Hermitian $2$-form $\hat{Q}^\lambda_m$ is equal to the number
of $s_j$ equal to $+1$ (resp. $-1$) arising in blocks of dimension $2m$  in the normal decomposition of $A$ on $V_{[\lambda]}$ given in theorem \ref{normalforms3}.
\end{proposition}

\begin{proof}
 On the intersection of  $ \Ker\bigl( (A-\lambda\id)^{m}\bigr)$
with one of the symplectically orthogonal   subspaces  $E^v_\lambda\oplus E^{\overline{v}}_{\overline{\lambda}}$
constructed above  from a $v$ such that $(A-\lambda\id)^p v\neq 0$ and $(A-\lambda\id)^{p+1}v= 0$,
 the form $\hat{Q}^\lambda_m$ vanishes identically, except if $p=m-1$ and the only non vanishing component  is
$\hat{Q}^\lambda_m(v,v)=\s$.\\
 Indeed, $\Ker\bigl( (A-\lambda\id)^{m}\bigr) \cap E_\lambda^v$ is spanned by 
 \[
 \{ (A-\lambda\id)^{r} v\, ;\, r\ge 0 \textrm{ and } r+m>p\, \},
 \]
  and
$\hat{Q}^\lambda_m\bigl((A-\lambda\id)^{r} v, (A-\lambda\id)^{r'} v\bigr)=0$ when $m+r+r'-1>p$  
so the only non vanishing cases arise when 
$r=r'=0$ and $m=p+1$ so for $\hat{Q}^\lambda_m(v,v)$.  This  is equal to
$\frac{1}{ \lambda} \Omega\bigl((A-\lambda\id)^{k}v,(A-\overline{\lambda}\id)^{k-1}\overline{v}\bigr)=\frac{1}{ \lambda} \lambda \s=\s$ if $m=2k$,
and to $ i\, \Omega\bigl((A-\lambda\id)^{k}v,(A-\overline{\lambda}\id)^{k}\overline{v}\bigr)=i (-i\s)=\s$ if $m=2k+1$.
\end{proof}
The numbers $s_j$ appearing in the decomposition are thus invariant of the matrix.
\begin{cor}\label{cor:unS1}
The normal decomposition described in theorem \ref{normalforms3} is unique up to a permutation of the blocks
when the eigenvalue $\lambda$ has been chosen in $\{ \lambda,\overline{\lambda}\}$, for instance by specifyng
that its imaginary part is positive. It is completely determined by this chosen $\lambda$, by the dimension 
$\dim_\C\bigl( \Ker (A-\lambda\Id)^r\bigr)$ for each $r\ge 1$ and by the rank and the signature of the
Hermitian bilinear $2$-forms  $\hat{Q}^\lambda_m$ for each $m\ge 1$.  \hfill{$\square$}
\end{cor}


\end{document}